\documentclass[border=1mm]{amsart}
\usepackage[utf8]{inputenc}
\usepackage{amsmath}
\usepackage{lipsum}
\usepackage{graphicx}
\usepackage{amsmath}
\usepackage{amsmath,kbordermatrix,blkarray}
\renewcommand{\kbldelim}{(}
\renewcommand{\kbrdelim}{)}

\usepackage{tikz}
\usepackage{stackengine}
\usepackage{amssymb}
\usepackage{mathtools}
\usepackage{blkarray, bigstrut}
\usepackage{amsthm}
\usepackage{enumitem}
\usepackage{amsthm}
\usepackage{upgreek}
\usepackage[english]{babel}
\usepackage{wasysym}
\usepackage{blkarray}
\usepackage{textcomp}
\usepackage{pifont}
\usepackage{enumitem} 
\usepackage{caption}
\usepackage{url}
\usepackage{float}

\usepackage{mathtools}

\usepackage{ textcomp }

\newtheorem{theorem}{Theorem}[section]
\newtheorem{corollary}{Corollary}[section]
\newtheorem{Lemma}[theorem]{Lemma}

\newtheorem{obs}{Observation}

\title[Fusions of the generalized Hamming scheme]{Fusions of the generalized Hamming scheme on a strongly-regular graph}

{\author[A.~Herman]{Allen Herman\textsuperscript{1,*}}}
\thanks{\textsuperscript{*} Research supported in part by an NSERC Discovery Research Grant, Application No.: RGPIN-2017-05331.}
\address{\textsuperscript{1} Department of Mathematics and Statistics, University of Regina, Regina, Saskatchewan S4S 0A2, Canada}
\email{Allen.Herman@uregina.ca}
{\author[N.~Joshi]{Neha Joshi\textsuperscript{2}}}
\address{\textsuperscript{2} Department of Mathematics and Statistics, University of Regina, Regina, Saskatchewan S4S 0A2, Canada}
\email{Njp008@uregina.ca}

\author[K.~Meagher]{Karen Meagher\textsuperscript{3,**}}
\address{\textsuperscript{3} Department of Mathematics and Statistics, University of Regina, Regina, Saskatchewan
S4S 0A2, Canada}
\email{karen.meagher@uregina.ca} 
\thanks{\textsuperscript{**} Research supported in part by an NSERC Discovery Research Grant,
    Application No.: RGPIN-2018-03952.}
\thanks{The authors have no relevant financial or non-financial interests to disclose.}

\begin{document}

\begin{abstract}
 In this paper we show that for any fusion $\mathcal{B}$ of an association scheme $\mathcal{A}$, the generalized Hamming scheme $H(n,\mathcal{B})$ is a nontrivial fusion of $H(n,\mathcal{A})$. We analyze the case where $\mathcal{A}$ is the association scheme on a strongly-regular graph, and determine the parameters of all strongly-regular graphs for which the generalized Hamming scheme, $H(2,\mathcal{A})$, has extra fusions, in addition to the one arising from the trivial fusion of $\mathcal{A}$. 
\end{abstract}



\maketitle

\section{Introduction}
The Hamming scheme $H(n,q)$ is one of the most important examples of an association scheme for coding theory. In 1992, Muzychuk \cite{HammingMuzychuk} proved that the Hamming scheme is \textit{fusion-primitive} (that is, has no nontrivial proper mergings) when $q\geq 4$. Generalized Hamming schemes were defined by Delsarte \cite{Delsarte} in 1973, where they were referred to as \textit{extensions}.
For a proof that the generalized Hamming scheme $H(n,\mathcal{A})$ is an association scheme for all choices of association scheme $\mathcal{A}$, see Godsil~\cite{HammingGodsil}.

 This work is motivated by a desire to understand the fusions of classical association schemes and their product schemes. In 1992, Sankey \cite{Sankey} gave the parameters of all pairs of strongly-regular graphs whose tensor product admits primitive strongly-regular fusions. Sankey's work addressed the question of the specific case of strongly-regular graphs, which are equivalent to homogeneous rank $3$ coherent configurations given by Higman~\cite{Higman}. Our work addresses the special case of the tensor product of a rank 3 scheme $\mathcal{A}$ with itself. We have found all possible fusions of $\mathcal{A} \otimes \mathcal{A}$ that exist for all $\mathcal{A}$, that is the fusion exists regardless of the eigenvalues of $\mathcal{A}$. The generalized Hamming scheme is one of these fusions. In the cases where the existence of a fusion depends on the parameters of $\mathcal{A}$, we give the required parameters of $\mathcal{A}$ and find examples of strongly-regular graphs admitting these parameters. We should also remark here that the same techniques can be used to determine the parameters of symmetric rank $3$ table algebras or character algebras whose symmetric tensor algebra will have extra fusions.
 
 In Section \ref{Sec:Preliminaries}, we review the basic theory of commutative association schemes and strongly-regular graphs to provide enough background necessary for later sections of this paper. In Section \ref{Sec:GHS3}, we prove that the generalized Hamming scheme is never fusion-primitive because it always has a nontrivial homogeneous fusion, $H(n,\mathcal{T}_{\mathcal{A}})$. We give details about the Bannai-Muzychuk criterion in Section \ref{Sec:4BM} which we later use to rule out many possible fusions of $H(2,\mathcal{A})$.
 
 In Section~\ref{sec:classifying}, we show that non-trivial fusions, other than the homogeneous fusion of $H(2,\mathcal{A})$, only exist in seven exceptional cases, and in each case we identify the strongly-regular graphs with the required parameters. The main result of this paper, formulated in Theorem \ref{mainresult}, provides a complete classification of fusions of the symmetric square of $\mathcal{A}$. We note that the result is not new, as it was obtained in 1985 by I.A. Farad{\u z}ev in \cite{FadRussian}, a translation is included in~\cite{FadRussianTranslated} .
The results obtained in Section \ref{sec:classifying} coincide with the ones presented in \cite{FadRussian}. Although our results are not new, the context and proof technique is completely different from the one used in \cite{FadRussian}. Specifically, our approach is based on character tables of association schemes.

\section{Preliminaries}\label{Sec:Preliminaries}

A set of $n\times n \ (0,1)$-matrices, $\mathcal{A}=\{A_0,A_1,\dots,A_d\}$ is the set of adjacency matrices of a \textit{commutative association scheme} if it satisfies the following properties;
 
\begin{enumerate}
    \item $A_0= I$.
    \medskip
    \item $\sum_{i=0}^{d} {A_i}= J$.
    \medskip
    \item $A_i^{T} \in \ \mathcal{A}, \textrm{for all} \ i=0,\dots,d$.
    \medskip
    \item $A_{i}A_{j}=\sum\limits_{\ell=0}^{d} p_{ij}^{\ell}A_{\ell}$ for all $i,j=0,\dots,d$.
    \medskip
    \item $A_{i}A_{j}=A_{j}A_{i}$ for all $i,j= {0,\dots,d}$.
\end{enumerate}
\medskip

If $(A_{i})^{T}=A_{i}$ for all $i=0,\dots,d$, then $ \{A_1,\dots,A_d \}$ are the adjacency matrices of graphs $ \{X_1,\dots,X_d \}$ and the scheme is called \textit{symmetric}. These graphs are called \textit{graphs of the scheme}. The nonnegative integers 
$p_{ij}^{\ell}$ are known as the \textit{structure constants} of the association scheme. It is also well known that the symmetry of the scheme asserts that $p_{ij}^{\ell}=p_{ji}^{\ell}$. One can find other interesting examples of an association scheme in Bailey~\cite{Bailey}. 

The matrices of an association scheme generate a commutative algebra, $\mathcal{A}$, with dimension $d+1$, of symmetric matrices with constant diagonal. This algebra is called the \textit{Bose-Mesner algebra} of $\mathcal{A}$ \cite{BoseMesner}. 
It is an algebra with respect to the usual matrix product as well as to the \textit{Hadamard} (or \textit{Schur}) \textit{product}, defined for two matrices $A$ and $B$ of order $n$ as the componentwise product: $(A \circ B)_{ij}=A_{ij}B_{ij}$. This algebra is also commutative and associative relative to this product with unit $J_n$.

\subsection{Fusions of an association scheme}

We can also define the set of adjacency matrices of an association scheme as a set of relations on the elements of $V$. We say that two vertices $v$, $w$ in $V$ are in relation $R_{i}$ if the $(v,w)^{th}$ entry of adjacency matrix $A_{i}$ equals $1$. We can denote an association scheme $\mathcal{A}$ as $({V}, \{{R_{i}\}}_{i=0}^{d})$. If $\mathcal{B}$ is an association scheme that can be expressed as $({V}, \{{R'_{j}\}}_{j=0}^{d'})$ where each relation $R'_{j}$ of $\mathcal{B}$ is a union of the relations in $\mathcal{A}$, then $\mathcal{B}$ is said to be a \textit{fusion} of $\mathcal{A}$.

For any association scheme $\mathcal{A}=({V}, \{{R_{i}\}}_{i=0}^{d})$ there always exists two fusions, namely the \textit{trivial fusion} $({V}, \{ R_{0}, \bigcup\limits_{i=1}^{d} R_{i}\} )$ and the scheme itself. A fusion is said to be a \textit{proper} fusion if its rank is strictly less than the rank of $\mathcal{A}$.  An association scheme is said to be \textit{fusion-primitive} if its only proper fusion is the trivial fusion.

Assume $\mathcal{B}=({V}, \{{R'_{j}\}}_{j=0}^{d'})$ is a fusion of $\mathcal{A}$. Let $\mathcal{I}\subseteq\{1,\dots,d\}$ and denote $R_{\mathcal{I}}=\bigcup\limits_{i\in \mathcal{I}} R_{i}$. If $R_{\mathcal{I}}$ is a relation in $\mathcal{B}$ then we will say that $\mathcal{B}$ is an \textit{$R_{\mathcal{I}}$-isolating fusion}. In terms of adjacency matrices, this means $A_{\mathcal{I}}=\sum\limits_{i\in \mathcal{I}} A_{i}$ is a matrix in $\mathcal{B}$.

\subsection{Hamming scheme and strongly-regular graphs}

Consider a set $Q$ with $q$ elements. The vertices of the \textit{Hamming scheme} are $n$-tuples of elements of $Q$. Any two $n$-tuples $\alpha$, $\beta$ are in the $i^{th}$ relation if they differ in exactly $i$ positions (they are known to be at \textit{Hamming distance} $i$). The Hamming scheme $H(n,q)$ is an association scheme on $q^{n}$ vertices of rank $n+1$ (Muzychuk~\cite{HammingMuzychuk}).

A \textit{strongly-regular graph} is a regular graph with the property that for any two distinct vertices $x$ and $y$ have either $\lambda$ or $\mu$ common neighbors according to whether $x$ and $y$ are adjacent or not. These are called the \textit{parameters} of the strongly-regular graph.
A strongly-regular graph  $\mathrm{X}$ produces a rank $3$ symmetric association scheme, $\mathcal{A}=\{I,A(\mathrm{X}),A(\overline{\mathrm{X}})\}$ where $I$ is the identity relation, $A(\mathrm{X})$ is the adjacency matrix of the graph $\mathrm{X}$ and $A(\overline{\mathrm{X}})$ the adjacency matrix of the complement graph $\overline{\mathrm{X}}$.

Let $\mathrm{X}$ be a strongly-regular graph with parameters $(n,k,\lambda,\mu)$. The spectrum of $\mathrm{X}$ is $k$, $r\geq0$ (with multiplicity $f$), $s<0$ (with multiplicity $g$). We set $\ell=n-k-1$ and denote the table of eigenvalues of adjacency matrices in $\mathcal{A}$ (also known as the \textit{character table} of $\mathcal{A}$) as,

\[
\mathcal{P}(\mathcal{A}) = 
    \begin{blockarray}{cccc}
      \begin{block}{[ccc]c}
       1 & k & \ell & 1\\
       1 & r & -1-r  & f\\
    1 & s & -1-s  & g\\
      \end{block}
    \end{blockarray}.
\]
\medskip

The orthogonality relation for rows $2$ and $3$ of this character table implies $$1 + \frac{rs}{k} +\frac{(-1-r)(-1-s)}{\ell} = 0,$$ so we have $$s = -\frac{(k+kr+k\ell)}{(k+kr+r\ell)}.$$

\vspace{0.2cm}
We can also obtain the multiplicities from the eigenvalues using the inner products of rows in the character table as
$$1+\frac{r^{2}}{k}+\frac{(-1-r)^{2}}{\ell}=\frac{n}{f}$$
and
$$1+\frac{s^{2}}{k}+\frac{(-1-s)^{2}}{\ell}=\frac{n}{g}.$$

\vspace{0.2cm}
A strongly-regular graph $\mathrm{X}$ is \textit{imprimitive} if either $\mathrm{X}$ or $\overline{\mathrm{X}}$ is disconnected. If $\mathrm{X}$ is disconnected then, $\mu=0$, hence $\mathrm{X}=\cup{K_{n}}$ and $s=-1$. If $\overline{\mathrm{X}}$ is disconnected  then $\mathrm{X}$ is a complete multipartite graph and $r=0$. A strongly-regular graph is primitive if $r>0$ and $s<-1$. Any primitive strongly-regular graph is
connected, and hence $k>r$ (Brouwer~\cite{Brouwer}). A  rank $3$ \textit{imprimitive association scheme} is an association scheme defined on the vertices of an imprimitive strongly-regular graph or its complement. The character table of a rank $3$ imprimitive association scheme is

\[
\mathcal{P}(\mathcal{A}) = 
    \begin{blockarray}{cccc}
      \begin{block}{[ccc]c}
     1 & k & m(1+k)  & 1\\
       1 & k & -1-k & m\\
    1 & -1 & 0  & k(1+m)\\
      \end{block}
    \end{blockarray}.
\]

The following inequalities, known as \textit{the Krein conditions}, are widely used as feasibility conditions for the eigenvalues $k, r$ and $s$ of strongly-regular graphs. For more details about the Krein parameters, see Mano, Martins and Vieira~\cite{Scott}.

\begin{Lemma}\textbf{Nonexistence: The Krein Conditions}. 
If {${\mathrm{X}}$} is a strongly-regular graph with parameters as defined above, then each of the following inequalities hold:
\begin{equation}\tag{$Krein_{1}$}
    (r+1)(k+r+2rs)\leq (k+r)(s+1)^{2},
\end{equation}
\begin{equation}\tag{$Krein_{2}$}
    (s+1)(k+s+2rs)\leq (k+s)(r+1)^{2}.
\end{equation}
\end{Lemma}

\subsection{Tensor product scheme}

The \textit{tensor product scheme}, or \textit{Kronecker product}, of any two association schemes of rank $d$ and $d'$ with relations $A_{i}$ and $B_{j}$ respectively is the association scheme whose relations are $A_{i} \otimes B_{j}$ where $0\leq i \leq d$ and $0 \leq j \leq d'$.

In Section $5$, we see that the generalized Hamming scheme $H(2,\mathcal{A})$ is always a non-trivial proper fusion of the tensor product scheme $\mathcal{A}\otimes \mathcal{A}.$


\section{Generalized Hamming Scheme}\label{Sec:GHS3}

Let $\mathcal{A}$ be an association scheme with $d$ classes and vertex set $V$. Assume the order of $V$ is $q$. If $v$ and $w$ are elements of $V^{n}$, let $h(v,w)$ be the vector of length $d+1$ with $r^{th}$-entry equal to the number of coordinates $j$ such that $v_{j}$ and $w_{j}$ are $r$-related in $\mathcal{A}$. For any $n$-tuples $v$ and $w$ the vector $h(v,w)$ has non-negative integer entries, and these entries sum to $n$. Conversely, any such vector can be written as $h(v,w)$ for some $v$ and $w$. 

Define a set $U$ consisting of length $d+1$ vectors, $h(v,w)$ with entries defined as above. The number of such vectors equals the order of $U$, so  $|U|=\binom{n+d}{d}$. Let $x$ be a non-negative integer vector of length $d+1$ with entries summing to $n$, let $A_{x}$ be the $01$-matrix with rows and columns indexed by $V^{n}$. The $(v,w)$-entry of $A_{x}$ equals to $1$ if and only if $h(v,w)=x$.

We denote the set of matrices, $\{A_{x}, \ x\in U\}$ by $H(n,\mathcal{A})$. $H(n,\mathcal{A})$ is an association scheme with vertex set $V^{n}$ and relations $R_{x}$, where $R_{x}$ is the relation on $V^{n}$ with adjacency matrix $A_{x}$. This is known as the \textit{generalized Hamming scheme} on the scheme $\mathcal{A}$. The vertices of generalized Hamming scheme consists of $n$-tuples with entries from $V$, so it is of order $q^{n}$. The rank of $H(n,\mathcal{A})$ is the number of non-negative integer vectors of length $d+1$ with entries summing to $n$, $\binom{n+d}{d}$.  It is easy to observe that if $\mathcal{A}=H(1,q)$ then $H(n,\mathcal{A})$ is the Hamming scheme $H(n,q)$ and $H(1,\mathcal{A})=\mathcal{A}$. The next result is from Godsil~\cite{HammingGodsil}.

\begin{theorem}\label{thm:godsil}
If $\mathcal{A}$ is an association scheme, then $H(n,\mathcal{A})$ is an association scheme.
\end{theorem}

\begin{theorem}\label{FusionGH}
Let $\mathcal{A}=({V}, \{{R_{i}\}}_{i=0}^{d})$ be a $d$-class association scheme and $\mathcal{B}=({V}, \{{R'_{j}\}}_{j=0}^{d'})$ be any arbitrary fusion of $\mathcal{A}$. Then $H(n,\mathcal{B})$ is a non-trivial fusion of $H(n,\mathcal{A})$.
\end{theorem}

\begin{proof}

The rows and columns of the matrices of the generalized Hamming scheme, $H(n,\mathcal{A})=\{C_{0},C_{1},\dots,C_{|U|}\}$ are labeled by the elements of $V^{n}$. For any $x$ in $U$, the $(v,w)$-entry of matrix $C_{x}$ is defined as
\[ C_{x}(v,w)= \begin{cases} 
          1 & if \  h(v,w)=x, \\
          0 & otherwise. 
       \end{cases}
    \]
We denote the vector corresponding to the identity relation as $e=(n,0,\dots,0)$. We are also given that the relations in $\mathcal{B}$ are a union of the relations of $\mathcal{A}$. Define the set $U'$ consisting of length $d'+1$ vectors, $h'(v,w)$ that sum to $n$. Similar to $H(n,\mathcal{A})$, the adjacency matrices of $H(n,\mathcal{B})=\{C'_{0},C'_{1},\dots,C'_{|U'|}\}$ are labeled by the elements of $V^{n}$. For any $y$ in $U'$, the $(v,w)^{th}$ entry of matrix $C'_{y}$ is defined as
\[ C'_{y}(v,w)= \begin{cases} 
          1 & if \  h'(v,w)=y, \\
          0 & otherwise. 
       \end{cases}
    \]
From Theorem~\ref{thm:godsil}, we know that $H(n,\mathcal{B})$ is an association scheme.
In order to prove that $H(n,\mathcal{B})$ is a fusion of $H(n,\mathcal{A})$, we only need to show that all the relations of $H(n,\mathcal{B})$ are a union of relations in $H(n,\mathcal{A})$.

We define a surjection $\phi: U \rightarrow U'$ such that for each $R'_{j}=\bigcup\limits_{i=1}^{i_{j}}R_{j_{i}}$, the vector $u\in U$ is mapped to a vector $u'\in U'$ with the entries of $u$ in positions $j_{1},j_{2},\dots,j_{i_{j}}$ summed together to produce the entry of $u'$ in position $j$, for $j=0,\dots,d'$. For each $C'_{y}$ in $U'$ there is a unique subset $\phi^{-1}(y)$ for which $C'_y = \sum_{x \in \phi^{-1}(y)} C_x$. 

In other words, the adjacency matrices $C_{x_{k}}$ corresponding to such relations in $H(n,\mathcal{A})$ are added to obtain the new relation $C'_{y_{k}}$ in $H(n,\mathcal{B})$. Since, $\mathcal{B}$ is already an association scheme, hence we can conclude that $H(n,\mathcal{B})$ will always be a non-trivial fusion of $H(n,\mathcal{A})$.

\end{proof}

For a fixed association scheme $\mathcal{A}$, we will denote the proper trivial fusion of $\mathcal{A}$ by $\mathcal{T}_{\mathcal{A}}$. It is clear that $d'=1$ for $\mathcal{T}_{\mathcal{A}}$, so $H(n,\mathcal{T}_{\mathcal{A}})$ has rank $\binom{n+1}{1}  = n+1$. From the above result, $H(n,\mathcal{T}_{\mathcal{A}})$ is always a non-trivial fusion of $H(n,\mathcal{A})$. We will refer to $H(n,\mathcal{T}_{\mathcal{A}})$ as the \textit{homogeneous fusion} of the generalized Hamming scheme. For any scheme $\mathcal{A}$, the generalized Hamming scheme will always have the homogeneous fusion.

\begin{corollary}
For all $n>1$, $H(n,\mathcal{A})$ is never fusion-primitive for any scheme $\mathcal{A}$.
\end{corollary}


\section{Calculating fusions from the character table}\label{Sec:4BM}

In this section, we give a very useful condition on the character table of an association scheme to have a nontrivial fusion.
Let $P$ be the $(d+1)\times (d+1)$-character table of a $d$-class association scheme. Let $\uptau=\{T_{0}=\{0\},T_{1},\dots,T_{d'}\}$ be a partition of the index set $\{0,\dots,d\}$. We define $P_{\uptau}$ to be the order $|\uptau|\times (d+1)$ matrix with the rows indexed by the classes $T \in \uptau$. The row corresponding to a class $T\in \uptau$ is the sum of rows in $P$ indexed by $t$ for all $t\in T$. Since $P$ is a non-singular matrix, the rank of $P_{\uptau}$ is equal to $|\uptau|$. Hence, the number of distinct columns in $P_{\uptau}$ is at least $|\uptau|$. The next Lemma, known as the \textit{Bannai-Muzychuk criterion} (see Muzychuk \cite{JohnsonMuzychuk}), states a necessary and sufficient condition for a partition $\uptau$ of the rows of $P$ to produce a fusion.

\begin{Lemma}{\label{Bannai-Muzychuk}}
A partition $\uptau=\{T_{0}=\{0\},\dots,T_{d'}\}$ of the index set $\{0,\dots,d\}$ determines a fusion scheme $\mathcal{B}=(V, \{{R'_{i}\}}_{i=0}^{d'})$ where $R'_{i}=\bigcup\limits_{j \in T_{i}}R_{j}$ if and only if the number of different columns in $P_{\uptau}$ equals $|\uptau|$.

\end{Lemma}

One can find more examples on the existence and construction of a  fusion from an association scheme using this criterion in Bannai~\cite{Bannai} and Bannai and Song~ \cite{BannaiSong}.

\section{Classifying fusions of $H(2,\mathcal{A})$}
\label{sec:classifying}

In this section, we classify all possible fusions of the generalized Hamming scheme $H(2,\mathcal{A})$ where $\mathcal{A}=\{I,A(\mathrm{X}),A(\overline{\mathrm{X}})\}$. In this case, $\mathcal{A}$ is a rank $3$ association scheme and both $A(\mathrm{X})$ and $A(\overline{\mathrm{X}})$ are adjacency matrices of strongly-regular graphs. Throughout this section, we will assume that the eigenvalues of $A(\mathrm{X})$ are $k$, $r$ and $s$ with $r>s$. Let $\ell$ denote the degree of the graph, $\overline{\mathrm{X}}$.

We use the eigenvalue table of the tensor product association scheme of $\mathcal{A}$ with itself, $\mathcal{A}^{\otimes 2}=\{C_{1},\dots,C_{9}\}$. The corresponding values of $C_{i}$ in terms of tensor products given by Sankey~\cite{Sankey} are,

\[
\begin{tabular}{ccc}
$C_{1}=I\otimes I$, &  $C_{2}=I\otimes A(\mathrm{X})$, & $C_{3}=I\otimes A(\overline{\mathrm{X}})$,\\
    $C_{4}=A(\mathrm{X})\otimes I$, &  $C_{5}=A(\mathrm{X})\otimes A(\mathrm{X})$, & $C_{6}=A(\mathrm{X})\otimes A(\overline{\mathrm{X}})$,\\
$C_{7}=A(\overline{\mathrm{X}})\otimes I$, & $C_{8}=A(\overline{\mathrm{X}})\otimes A(\mathrm{X})$, &  $C_{9}=A(\overline{\mathrm{X}})\otimes A(\overline{\mathrm{X}})$.\\
\end{tabular}
\]

If we exchange $\mathrm{X}$ for $\overline{\mathrm{X}}$, then we get the same matrices, but in a different order. For example $C_2$ is switched with $C_3$; we will call these pairs \textit{switch partners}.

\medskip

The generalized Hamming scheme, $$H(2,\mathcal{A})=\{C_{1},\ C_{5},\ C_{9},\ C_{2}+C_{4},\ C_{3}+C_{7},\ C_{6}+C_{8}\},$$ is one of the fusions of the rank $9$ association scheme, $\mathcal{A}\otimes \mathcal{A}$. 
The character table of $H(2,\mathcal{A})$ is given below. We label the $6$ irreducible characters by $\chi_{i}$, $i=1,\dots,6$. The eigenvalues on $C_{1}$ are omitted. 
 
\[
 \mathcal{P}(H(2,\mathcal{A})) = 
    \begin{blockarray}{cccccc}
         C_{5} & C_{9} & C_{2}+C_{4} & C_{3}+C_{7} & C_{6}+C_{8} \\
      \begin{block}{[ccccc]c}
       k^2 & \ell^2 & 2k & 2\ell & 2k\ell & \chi_{1} \\
      kr & -\ell(1+r) & k+r & \ell-1-r & -k(1+r)+\ell r & \chi_{2}  \\
      ks & -\ell(1+s) & k+s & \ell-1-s & -k(1+s)+\ell s & \chi_{3} \\
      r^2 & (-1-r)^2 & 2r & -2(1+r) & -2r(1+r) & \chi_{4} \\
      rs & 1+r+s+rs & r+s & -2-r-s & -r-s-2rs & \chi_{5} \\
      s^2 & (-1-s)^2 & 2s & -2(1+s) & -2s(1+s) & \chi_{6} \\
      \end{block}
    \end{blockarray}.
\]

A complete list of the fusions of $\mathcal{A}\otimes \mathcal{A}$ that exist for all rank $3$ association schemes $\mathcal{A}$ is given in Section 6 of this article. 
\subsection{Homogeneous fusion}
In order to classify all possible non-trivial fusions of $H(2,\mathcal{A})$, we need to check all partitions of $S\subseteq \{2,\dots,6\}$ for possible fusions using Lemma \ref{Bannai-Muzychuk}. There are 52 cases to check. The problem reduces to finding conditions on the eigenvalues $k, \ \ell, \ r$ and $s$  for which the given partition satisfies the condition in Lemma \ref{Bannai-Muzychuk}.

\begin{corollary}\label{homofusion}
$H(2,\mathcal{A})$ will always have the homogeneous fusion that comes from the trivial fusion $\mathcal{T}_{\mathcal{A}}$ of $\mathcal{A}$, namely $$H(2,\mathcal{T}_{\mathcal{A}})=\{C_{1}, \ (C_{2}+C_{4})+(C_{3}+C_{7}), \ C_{5}+C_{9}+(C_{6}+C_{8})\}.$$
\end{corollary}

\begin{proof}
This result is an extension of Theorem \ref{FusionGH} for the special case when $n=2$ and $\mathcal{A}$ has rank $3$.
The character table of $H(2,\mathcal{T}_{\mathcal{A}})$ is

\[
\mathcal{P}(H(2,\mathcal{T}_{\mathcal{A}})) = 
    \begin{blockarray}{cccc}
      \begin{block}{[ccc]c}
       1 & 2(k+\ell) & (k+\ell)^{2} \bigstrut[t] & 1\\
       1 & k+\ell-1 & -(k+\ell)  & 2(k+\ell)\\
    1 & -2 & 1  & (k+\ell)^{2}\\
      \end{block}
    \end{blockarray}.
\]

\end{proof}
It is clear that the homogeneous fusion does not depend on the eigenvalues $k, \ \ell$ and $r$. The structure is entirely defined by the number of vertices, $n$ of the strongly regular graph $\mathrm{X}$. We can also refer to the homogeneous fusion as $H(2,K_{n})$ because $\mathcal{T}_{\mathcal{A}}$ is the rank $2$ association scheme. The adjacency matrices of the graphs in $H(2,\mathcal{T}_{\mathcal{A}})$ can be expressed as $$\{I\otimes I, \ I\otimes (J-I)+(J-I)\otimes I, \  (J-I)\otimes (J-I)\}$$
where the second matrix is the adjacency matrix of the cartesian product of two complete graphs of order $n$.

\subsection{Character tables admitting specific fusions}

Our main result shows that for most strongly-regular graphs $X$, the only proper fusions of $H(2,\mathcal{A})$ will be the trivial and the homogeneous fusion. We will characterize all strongly-regular graphs for which $H(2,\mathcal{A})$ has additional fusions. We have used equations in Section $2.2$ to determine the character table in one variable. Existence of extra fusions requires special conditions on eigenvalues.

The next few results review these special conditions. We will first consider the case of the imprimitive association scheme.


\begin{theorem}\label{Completegraphs}

\begin{enumerate}[label=(\roman*)]
If $\mathrm{X}$ is either a union of complete graphs, or a complete multipartite graph, then the character table for the association scheme is

\[
\mathcal{P}(\mathcal{A})= 
    \begin{blockarray}{cccc}
      \begin{block}{[ccc]c}
        1 & r & m(1+r) & 1\\
       1 & r & -1-r  & m\\
    1 & -1 & 0  & r(1+m)\\
      \end{block}
    \end{blockarray}.
\]

    \item If $\mathrm{X}$ is the union of complete graphs implying $k=r$ and $s=-1$, then there are additional proper non-trivial fusions other than the homogeneous fusion of $H(2,\mathcal{A})$ namely:
    \vspace{0.01cm}
\begin{enumerate}[label={}]
  \item $\{C_{1}, \ C_{2}+C_{4}, \ C_{3}+C_{7}+C_{6}+C_{8}, \ C_{5}, \ C_{9}\}$,~~~~~~~~~~~~~~~~~~~~~~~~~~~~~~~~~~{(1)}
  \item $\{C_{1}, \ C_{2}+C_{4}, \ C_{3}+C_{7}+C_{6}+C_{8}+C_{9}, \ C_{5}\}$,~~~~~~~~~~~~~~~~~~~~~~~~~~~~~~~~~~{(2)}
      \item $\{C_{1}, \ C_{2}+C_{4}+C_{5}, \ C_{3}+C_{7}+C_{6}+C_{8},\ C_{9}\}$,~~~~~~~~~~~~~~~~~~~~~~~~~~~~~~~~~~{(3)}
      \item $\{C_{1}, \ C_{2}+C_{4}+C_{5}, \ C_{3}+C_{7}+C_{6}+C_{8}+C_{9}\}$.~~~~~~~~~~~~~~~~~~~~~~~~~~~~~~~~~~{(4)}
     
\end{enumerate}

\item If $\mathrm{X}$ is the complete multipartite graph implying $r=0$ and $\ell=-1-s$, then there are additional proper non-trivial fusions other than the homogeneous fusion of $H(2,\mathcal{A})$ which are the switch partners of part one above:

\begin{enumerate}[label={}]
  \item $\{C_{1}, \ C_{3}+C_{7}, \ C_{2}+C_{4}+C_{6}+C_{8}, \ C_{5}, \ C_{9}\}$,~~~~~~~~~~~~~~~~~~~~~~~~~~~~~~~~~~{(1')}
  \item $\{C_{1}, \ C_{3}+C_{7}, \ C_{2}+C_{4}+C_{6}+C_{8}+C_{5}, \ C_{9}\}$,~~~~~~~~~~~~~~~~~~~~~~~~~~~~~~~~~~{(2')}
      \item $\{C_{1}, \ C_{3}+C_{7}+C_{9}, \ C_{2}+C_{4}+C_{6}+C_{8},\ C_{5}\}$,~~~~~~~~~~~~~~~~~~~~~~~~~~~~~~~~~~{(3')}
      \item $\{C_{1}, \ C_{3}+C_{7}+C_{9}, \ C_{2}+C_{4}+C_{6}+C_{8}+C_{5}\}$.~~~~~~~~~~~~~~~~~~~~~~~~~~~~~~~~~~{(4')}
\end{enumerate}

\end{enumerate}

\end{theorem}

\begin{proof}

The character table for this case is in two variables and we were able to build the character table for the generalized Hamming scheme. We applied Lemma $4.1$ to check for all possible fusions using GAP~\cite{GAP4}.  The code for implementation of Generalized Hamming scheme is open sourced at GitHub repository \footnote{\url{https://github.com/nehamainali/GeneralizedHammingScheme}}.

\end{proof}
In part $(i)$ of Theorem \ref{Completegraphs}, the graph $\mathrm{X}$ is the union of complete graphs of order $r+1$. This happens exactly when $k=r$ and $s=-1$. Its complement is the complete multipartite graph $K_{r\times m}$. This case happens exactly when $r=0$. In the second part of Theorem \ref{Completegraphs}, we have switched $\mathrm{X}$ with its complement and found the switch partner fusions. The above theorem holds for all rank $3$ imprimitive association schemes. For more properties of an imprimitive association scheme, see Chapter $9$ in Brouwer~\cite{Brouwer}. 

\medskip

Now that we have classified all possible fusions for the case where the strongly-regular graph, $\mathrm{X}$ or its complement is imprimitive. Therefore we will not consider the possible fusions when $s<-1$ and $k>r>0$ for all of the remaining special fusions in our list.


\begin{theorem}\label{News}
If $k=s^2$, $\ell=-2s$ with $r=1$, then the character table for the association scheme is

\[
\mathcal{P}(\mathcal{A})= 
    \begin{blockarray}{cccc}
      \begin{block}{[ccc]c}
        1 & s^2 & -2s & 1\\
       1 & 1 & -2  & s^2\\
    1 & s & -1-s  & -2s\\
      \end{block}
    \end{blockarray}.
\]

\begin{enumerate}[label=(\roman*)]
    \item If $\mathrm{X}$ is the cartesian product of two complete graphs with order $s+1$ implying $k=s^2$ and $r=1$, then there is an additional rank $5$ fusion of $H(2,\mathcal{A})$ namely,
    \begin{itemize}[label={}]
  \item $\{C_{1}, \ C_{2}+C_{4}+C_{9}, \ C_{3}+C_{7},\ C_{6}+C_{8}, \ C_{5} \}$.~~~~~~~~~~~~~~~~~~~~~~~~~~~~~~~~~~{(5)}
  \vspace{0.002cm}
  \end{itemize}
  \item If $\mathrm{X}$ is the complement of the cartesian product of two complete graphs with order $r+2$ implying $k=2(1+r)$, $\ell=(1+r)^{2}$ and $s=-2$, then there is an additional rank $5$ fusion of $H(2,\mathcal{A})$ which is the switch partner of part one above,
\begin{itemize}[label={}]
  \item $\{C_{1}, \ C_{2}+C_{4}, \ C_{9}, \ C_{3}+C_{7}+C_{5}, \  C_{6}+C_{8} \}$.~~~~~~~~~~~~~~~~~~~~~~~~~~~~~~~~~~{(5')}
\end{itemize}
\end{enumerate}

\end{theorem}
\begin{proof}

The character table for the first part of this case is in terms of a single parameter $s$ and we were able to apply Lemma $4.1$ on the character table of the generalized Hamming scheme. For the switch partner fusion, we used the conditions, $k=2(1+r)$ and $\ell=(1+r)^{2}$  to write $1+r$ and $2r$ in terms of $k$, $\ell$ and $s$. After substituting these values into the row orthogonality condition $-k(1+s)+s\ell=-2r(1+r)$, we get $s=-2$. This way we were able to express our character table in a single parameter $r$. The fusions above are verified via GAP~\cite{GAP4}.

\end{proof}
The association scheme on the cartesian product of two complete graphs with order $-s+1$ gives the same character table as in part $(i)$ of Theorem \ref{News}. Note that, this fusion is not the same as Corollary \ref{homofusion} as the fusion in Theorem \ref{News} is a special case fusion coming from a fewer vertices than $H(2,\mathcal{A})$ itself.


\begin{theorem}\label{Payley}
If $k=\ell$ and $s=-1-r$ then the character table $\mathcal{P}(\mathcal{A})$ has the form 

\[
\mathcal{P}(\mathcal{A})= 
    \begin{blockarray}{cccc}
      \begin{block}{[ccc]c}
        1 & 2(r+r^{2}) & 2(r+r^{2}) \bigstrut[t] & 1\\
1 & r & -1-r & 2(r+r^{2})\\
1 & -1-r & r & 2(r+r^{2})\\
      \end{block}
    \end{blockarray}.
\]

For this case there are additional proper non-trivial fusions other than the homogeneous fusion of $H(2,\mathcal{A})$ namely:
\begin{itemize}[label={}]
  \item $\{C_{1}, \ C_{2}+C_{4}+C_{3}+C_{7},\ C_{6}+C_{8}, \ C_{5}+C_{9}\}$,~~~~~~~~~~~~~~~~~~~~~~~~~~~~~~~~~~{(6)}
  \item $\{C_{1}, \ C_{2}+C_{4}+C_{3}+C_{7}+ C_{6}+C_{8}, \ C_{5}+C_{9}\}$,~~~~~~~~~~~~~~~~~~~~~~~~~~~~~~~~~~{(7)}
  \item $\{C_{1}, \ C_{2}+C_{4}+C_{3}+C_{7}+ C_{5}+C_{9}, \ C_{6}+C_{8}\}$.~~~~~~~~~~~~~~~~~~~~~~~~~~~~~~~~~~{(8)}
\end{itemize}

\end{theorem}
\begin{proof}
We applied the orthogonality conditions from Section $2$ to the conditions $k=\ell$ and $s=-1-r$ and obtained $k=\ell=2(r+r^2)$. Hence, our character table can be expressed in one variable. These fusions were checked using Lemma \ref{Bannai-Muzychuk} via GAP \cite{GAP4}.
\end{proof}
In the above theorem, the adjacency matrices $A(\mathrm{X})$ and $A(\overline{\mathrm{X}})$ are cospectral. One well-known example satisfying the above theorem is the Paley graph. These fusions are closed under switch partners, that is, we get the same fusions if we switch the graphs $\mathrm{X}$ and $\overline{\mathrm{X}}$.

\begin{theorem}\label{Crazyrank4}
If $k=3-s-r$ and $\ell=5+s+r$ then the character table of the association scheme, $\mathcal{A}$ has the form

\[
\mathcal{P}(\mathcal{A})= 
    \begin{blockarray}{cccc}
      \begin{block}{[ccc]c}
       1 & 3-s-r & 5+s+r  & 1\\
       1 & r & -1-r  & f\\
    1 & s & -1-s & g\\
      \end{block}
    \end{blockarray}.
\]

For this case there is an additional rank $4$ fusion of $H(2,\mathcal{A})$ namely,
\begin{itemize}[label={}]
  \item $\{C_{1}, \ C_{2}+C_{4}+C_{9}, \ C_{3}+C_{7}+C_{5}, \  C_{6}+C_{8} \}$.~~~~~~~~~~~~~~~~~~~~~~~~~~~~~~~~~~{(9)}
  
\end{itemize}
 \end{theorem}
 \begin{proof}
 We were able to generate the character table of the generalized Hamming scheme from $\mathcal{P}(\mathcal{A})$ and then apply Lemma \ref{Bannai-Muzychuk} to the eigenvalues to check for possible fusions. The above fusion has also been verified using GAP.
 \end{proof}

 The strongly-regular graphs satisfying the conditions of Theorem \ref{Crazyrank4} will have order exactly $9$.  There are two association schemes, but $3$ possible graphs, $\mathrm{X}$ with degree $2$, $4$, and $6$. The third graph with degree $6$ is the complement of the degree $2$ graph, and has $r=0$ and $s=-3$, so $k=6$ and $\ell=2$. We should also note that for the case when $s=-1$ and $r=2$, implying $k=2$ and $\ell=2(1+r)=2(3)=6$, the above theorem overlaps with Theorem \ref{Completegraphs} which holds for imprimitive graphs. Also if $r=1$ and $s=-2$ implying $k=\ell=4$ then Theorem \ref{Crazyrank4} overlaps with Theorem \ref{Payley}.

\begin{theorem}\label{Clebsch}
If $k=r(3+r)$ and $\ell=(3+r)$ with $s=-2$ then the character table, $\mathcal{P}(\mathcal{A})$ has the form

\[
\mathcal{P}(\mathcal{A})= 
    \begin{blockarray}{cccc}
      \begin{block}{[ccc]c}
       1 & r(3+r) & 3+r  & 1\\
1 & r & -1-r & (3+r)\\
1 & -2 & 1 & r(3+r)\\
      \end{block}
    \end{blockarray}.
\]

\begin{enumerate}[label=(\roman*)]
    \item For the above character table, we get a rank $3$ non-trivial fusion other than the homogeneous fusion of $H(2,\mathcal{A})$ namely,
\begin{itemize}[label={}]
  \item $\{C_{1}, \  C_{2}+C_{4}+C_{9},\ C_{6}+C_{8}+C_{5}+C_{3}+C_{7}\}$.~~~~~~~~~~~~~~~~~~~~~~~~~~~~~~~~~~{(10)}
\end{itemize}
\item The above fusion has a switch partner which gives the rank $3$ fusion, 
\begin{itemize}[label={}]
  \item $\{C_{1}, \  C_{3}+C_{7}+C_{5},\ C_{6}+C_{8}+C_{9}+C_{2}+C_{4}\}$.~~~~~~~~~~~~~~~~~~~~~~~~~~~~~~~~~~{(10')}
\end{itemize}
The graph $\mathrm{X}$ in this case satisfies the conditions, $\ell=-k(s+1)$ and $k=2-s$ with $r=1$.
\end{enumerate}

\end{theorem}
\begin{proof}
The above character table is expressed in one variable and hence we were able to check for possible fusions using Lemma \ref{Bannai-Muzychuk}. The above fusion has been verified via GAP \cite{GAP4}.
\end{proof}
Applying the Krein conditions (see Section 2.1) on the eigenvalues for the strongly-regular graph in Theorem \ref{Clebsch} shows that $r\leq 2$. When $r=1$ the graph $A(\mathrm{X})$ must be the Paley graph on $9$ vertices, since this graph is determined by its spectrum. When $r=2$, the spectrum of the complement graph, $\overline{\mathrm{X}}$ matches that of the Clebsch graph. Applying the Krein condition on the switch partner table will restrict $s$ to $-1$, $-2$, or $-3$. The graphs in part (ii) are the complements of the two graphs in part (i).

\begin{theorem}\label{Clebsh}
If $k-\ell =1+2r$ with $r=-s$ then the character table $\mathcal{P}(\mathcal{A})$ has the form

\[
\mathcal{P}(\mathcal{A})= 
    \begin{blockarray}{cccc}
      \begin{block}{[ccc]c}
        1 & r(2r+1) & (r-1)(2r+1) \bigstrut[t] & 1\\
1 & r & -1-r & r(2r+1)\\
1 & -r & -1+r & (r-1)(2r+1)\\
      \end{block}
    \end{blockarray}.
\]

\begin{enumerate}[label=(\roman*)]
    \item For the above character table, we get a rank $3$ non-trivial fusion other than the homogeneous fusion of $H(2,\mathcal{A})$ namely,
    
\begin{itemize}[label={}]
  \item $\{C_{1}, \ C_{2}+C_{4}+C_{6}+C_{8},\ C_{5}+C_{9}+C_{3}+C_{7}\}$.~~~~~~~~~~~~~~~~~~~~~~~~~~~~~~~~~~{(11)}
\end{itemize}
\item The above fusion has a switch partner which gives the rank $3$ fusion,  
\begin{itemize}[label={}]
  \item $\{C_{1}, \ C_{3}+C_{7}+C_{6}+C_{8},\ C_{5}+C_{9}+C_{2}+C_{4}\}$.~~~~~~~~~~~~~~~~~~~~~~~~~~~~~~~~~~{(11')}
  
\end{itemize}
The graph $\mathrm{X}$ in this case satisfies the conditions, $\ell-k=2r+3$ and $\ell=(r+1)(2r+3)$ with $s=-2-r$.

\end{enumerate}

\end{theorem}
\begin{proof}
We can apply the orthogonality conditions on the above character table. If $\ell=k-1-2r$ and $s=-r$ then $k=r(2r+1)$. Hence, the character table can be expressed in one variable. We were able to verify this fusion via GAP. 
\end{proof} 

Applying the Krein conditions on the eigenvalues for the strongly-regular graph in the first part of Theorem \ref{Clebsh} implies $r\geq 2$. When $r=2$, the association scheme is the Clebsch graph scheme, as in Theorem \ref{Clebsch}. When $r=3$, using Brouwer's online directory of strongly-regular graphs, we see that there are $180$ nonisomorphic $2$-graphs with this spectrum. When $r=4$, there is an affine polar graph with this spectrum. It is possible that strongly-regular graphs with the spectrum as in Theorem \ref{Clebsh} would exist for all $r\geq 2$. The first case where such a graph is not known to exist is for $r=7$.

\subsection{Our main result}

Our goal is to find all of the strongly-regular graphs, $\mathrm{X}$ for which $H(2,\mathcal{A})$ has a nontrivial fusion other than the homogeneous fusion.  We do this by determining parameter sets for $\mathcal{I}$-isolating fusions, where $\mathcal{I}$ is a $1$ or $2$ element subset of the nontrivial elements of $H(2,\mathcal{A})$.

If we assume that there is an $\mathcal{I}$-isolating fusion then the graph $\mathrm{X}_{\mathcal{I}}$ cannot have more than $d'+1$ eigenvalues where $d'$ is the rank of the fusion. By considering different sets $\mathcal{I}$, we can easily calculate the eigenvalues of $\mathrm{X}_{\mathcal{I}}$ and give a lower bound on the rank of an $\mathcal{I}$-isolating fusion of $H(2,\mathcal{A})$. These are listed in $19$ Observations. We list all possible rank $3$, $4$ and $5$ fusions of $H(2,\mathcal{A})$ in Lemma \ref{rank5}--\ref{rank3}. We have also listed the Observation used to rule out the existence of a particular fusion.

We will denote the eigenvalue corresponding to the $i^{th}$ character of the basis element $C_{j}$ as $\chi_{i}(C_{j})$. Here, $\chi_{1}(C_{j})$ refers to the trivial character (which gives the degree of the graph) which is always the biggest value in every column of the character table of $H(2,\mathcal{A})$, hence $\chi_{1}(C_{j}) \neq \chi_{i}(C_{j})$ for all $i,\ j=2,\dots,6$. 

Since we have already discussed the cases when $s=-1$ and $k=r$. We can assume that $k>r>0>-1>s$ and $\ell>-1-s>0>-1>-1-r$ for all of the Observations below.

\begin{obs}
 There is no possible rank $3$ or rank $4$ $C_{5}$-isolating fusion. A rank $5$ $C_{5}$-isolating fusion of $H(2,\mathcal{A})$ is possible if and only if exactly one of the equalities $kr=s^2$ or $r=-s$ hold. The condition $kr=s^{2}$ is equivalent to $\chi_{2}(C_{5})=\chi_{6}(C_{5})$ and $r=-s$ is equivalent to $\chi_{4}(C_{5})=\chi_{6}(C_{5})$.
\end{obs}

 \begin{proof}
 
 We must have the following column by itself in the new character table for any feasible $C_{5}$-isolating fusion.

   \[
\begin{blockarray}{cc}
\mathbf{\textit{C}_{5}} \\
\begin{block}{(c)c}
  k^2 & \chi_{1} \\
  kr & \chi_{2} \\
  ks & \chi_{3} \\
  r^2 & \chi_{4} \\
  rs & \chi_{5} \\
  s^2 & \chi_{6} \\
\end{block}
\end{blockarray}
 \]
 
 For illustration purposes, we draw a poset with points as entries of the above column and line segments are drawn between these points if and only if there is a strict inequality, with the entry of larger value above.
    
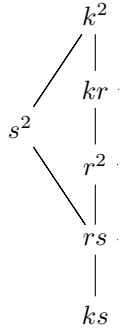
\begin{center}
         \begin{tikzpicture}[scale=.5]
  \node (one) at (0,6) {$k^2$};
  \node (1) at (0,0) {$rs$};
  \node (b) at (-2,3) {$s^2$};
  \node (c) at (0,2) {$r^2$};
  \node (d) at (0,4) {$kr$};
  \node (zero) at (0,-2) {$ks$};
  \draw (zero) -- (1) -- (1) -- (b) -- (1) -- (c) -- (c) --(d) -- (d) -- (one) -- (b) -- (one);
\end{tikzpicture}
    \end{center}

      The entries in this column that could be equal are: $\chi_{2}(C_{5})=\chi_{6}(C_{5})$ and $\chi_{4}(C_{5})=\chi_{6}(C_{5})$. This can only happen if and only if either $kr=s^2$ or $r=-s$. If both the equalities are satisfied, that is $\chi_{2}(C_{5})=\chi_{4}(C_{5})=\chi_{6}(C_{5})$ then $k=-s$ which is a contradiction since $k>|s|$.
    
  Hence we conclude that a rank $5$, $C_{5}$-isolating fusion is possible if exactly one of the equalities $r=-s$ or $kr=s^2$ hold. The rank $5$ fusion that comes from $kr=s^2$ is covered in Theorem \ref{News}.

  As it is clear from the length of the poset above, in any case we will get at least $5$ rows with distinct entries. Hence, there is no possible rank $3$ or rank $4$ $C_{5}$-isolating fusion using Lemma \ref{Bannai-Muzychuk}.
    
 \end{proof}

\medskip
\begin{obs}
    There is no possible rank $3$ or rank $4$ $C_{9}$-isolating fusion. A rank $5$ $C_{9}$-isolating fusion of $H(2,\mathcal{A})$ is possible if and only if exactly one of the equalities $r=-2-s$ or $\ell (-1-s)=(1+r)^2$ hold. The condition $r=-2-s$ is equivalent to $\chi_{4}(C_{9})=\chi_{6}(C_{9})$, and  $\ell (-1-s)=(1+r)^2$ is equivalent to $\chi_{4}(C_{9})=\chi_{3}(C_{9})$. 
    
\end{obs}
    
\begin{proof}    
    We must have the following column by itself 
in the new character table for any feasible $C_{9}$-isolating fusion.
  \[
\begin{blockarray}{cc}
\mathbf{\textit{C}_{9}} \\
\begin{block}{(c)c}
  \ell^2 & \chi_{1} \\
  -\ell(1+r) & \chi_{2} \\
  -\ell(1+s) & \chi_{3} \\
  (1+r)^2 & \chi_{4} \\
  (1+r)(1+s) & \chi_{5} \\
  (1+s)^2 & \chi_{6} \\
\end{block}
\end{blockarray}
 \]
 The poset for this column is
 
   
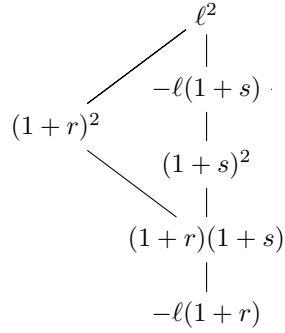
\begin{center}
         \begin{tikzpicture}[scale=.5]
  \node (one) at (0,6) {$ \ell ^{2}$};
  \node (1) at (0,0) {$(1+r)(1+s)$};
  \node (b) at (-4,3) {$(1+r)^2$};
  \node (c) at (0,2) {$(1+s)^2$};
  \node (d) at (0,4) {$- \ell (1+s)$};
  \node (zero) at (0,-2) {$-\ell (1+r)$};
  \draw (zero) -- (1) -- (c) (b) -- (1) (c)--(d) -- (d) -- (one) -- (b) -- (one);
\end{tikzpicture}
    \end{center}
  The only pair of entries in this column that could possibly be equal are $\chi_{4}(C_{9})=\chi_{6}(C_{9})$ or  $\chi_{4}(C_{9})=\chi_{3}(C_{9})$. This happens if and only if either $r=-2-s$ or $-\ell (1+s)=(1+r)^2$. If both the equalities are satisfied, that is $\chi_{3}(C_{9})=\chi_{4}(C_{9})=\chi_{6}(C_{9})$ then $\ell=-1-s$ which is a contradiction since $\ell>|-1-s|$ and is clear from the poset.
    
    Hence a rank $5$, $C_{9}$-isolating fusion is possible if either $\chi_{4}(C_{9})=\chi_{6}(C_{9})$ or $\chi_{4}(C_{9})=\chi_{3}(C_{9})$ implying either $r=-2-s$ or $-\ell (1+s)=(1+r)^2$ holds. The feasible rank $5$ that comes from the latter case is covered in Theorem \ref{News}.

    There are no possible $C_{9}$-isolating fusions of rank $3$ or $4$.
    
 \end{proof}
   \medskip 
    \begin{obs}
    There is no possible rank $3$ or rank $4$ $(C_{2}+C_{4})$-isolating fusion. A rank $5$ $(C_{2}+C_{4})$-isolating fusion of $H(2,\mathcal{A})$ is possible if and only if $k+s=2r$. The condition $k+s=2r$ is equivalent to $\chi_{3}(C_{2}+C_{4})=\chi_{4}(C_{2}+C_{4})$. 
    
\end{obs}
        
  \begin{proof}  
    We must have the following column by itself 
in the new character table for any feasible $(C_{2}+C_{4})$-isolating fusion.
  \[
\begin{blockarray}{cc}
\mathbf{\textit{C}_{2}+\textit{C}_{4}} \\
\begin{block}{(c)c}
  2k & \chi_{1} \\
  k+r & \chi_{2} \\
  k+s & \chi_{3} \\
  2r & \chi_{4} \\
  r+s & \chi_{5} \\
  2s & \chi_{6} \\
\end{block}
\end{blockarray}
 \]
 The poset for this column is
   
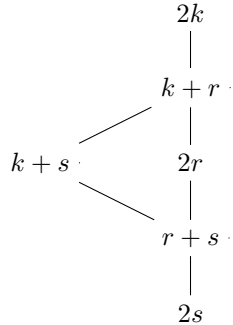
\begin{center}
         \begin{tikzpicture}[scale=.5]
  \node (one) at (0,6) {$2k$};
  
  \node (b) at (-4,2) {$k+s$};
  
  \node (d) at (0,4) {$k+r$};
  \node (c) at (0,2) {$2r$};
  \node (1) at (0,0) {$r+s$};
  \node (zero) at (0,-2) {$2s$};
  \draw (zero) -- (1) -- (1) -- (b) -- (b) --(d) -- (d) -- (one) (1) -- (c) -- (d);
\end{tikzpicture}
    \end{center}
  The only pair of entries that could be equal is $\chi_{3}(C_{2}+C_{4})=\chi_{4}(C_{2}+C_{4})$, which happens when $k+s=2r$. Hence, a rank $5$ fusion is possible if $\chi_{3}(C_{2}+C_{4})=\chi_{4}(C_{2}+C_{4})$ implying $k+s=2r$. 
  
  As clear from the height of our poset, there are no rank $3$ or $4$ fusions possible for this case.

\end{proof}
\medskip
     \begin{obs}
    There is no possible rank $3$ or rank $4$ $(C_{3}+C_{7})$-isolating fusion. A rank $5$ $(C_{3}+C_{7})$-isolating fusion of $H(2,\mathcal{A})$ is possible if and only if $\ell=r-2s-1$. This condition is equivalent to $\chi_{2}(C_{3}+C_{7})=\chi_{6}(C_{3}+C_{7})$. 
    
\end{obs}
          \begin{proof}
    We must have the following column by itself 
in the new character table for any feasible $(C_{3}+C_{7})$-isolating fusion.
    
  \[
\begin{blockarray}{cc}
\mathbf{\textit{C}_{3}+\textit{C}_{7}} \\
\begin{block}{(c)c}
  2\ell & \chi_{1} \\
  \ell-r-1 & \chi_{2} \\
  \ell-s-1 & \chi_{3} \\
  -2-2r & \chi_{4} \\
  -2-r-s & \chi_{5} \\
  -2-2s & \chi_{6} \\
\end{block}
\end{blockarray}
 \]
 The poset for this column is
   
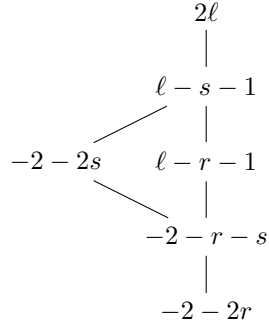
\begin{center}
         \begin{tikzpicture}[scale=.5]
  \node (one) at (0,8) {$2 \ell$};
  
  \node (b) at (0,4) {$\ell-r-1$};
  
  \node (d) at (0,6) {$\ell-s-1$};
  \node (c) at (-4,4) {$-2-2s$};
  \node (1) at (0,2) {$-2-r-s$};
  \node (zero) at (0,0) {$-2-2r$};
  \draw (zero) -- (1) -- (b)-- (d) -- (one) (d) -- (c) -- (1)
  ;
\end{tikzpicture}
    \end{center}
    
    The only pair of entries in this column that could possibly be equal are $\chi_{2}(C_{3}+C_{7})=\chi_{6}(C_{3}+C_{7})$ when $\ell=r-2s-1$. Hence a rank $5$ $C_{3}+C_{7}$-isolating fusion is possible if $\chi_{2}(C_{3}+C_{7})=\chi_{6}(C_{3}+C_{7})$ implying $\ell=r-2s-1$. This fusion is compatible with Observation $1$ and is covered in Theorem \ref{News}. 
    
    There are no possible $\textit{C}_{3}+\textit{C}_{7}$-isolating fusions of rank $3$ or $4$.
    
\end{proof}    
\medskip

By Observations $1$--$4$, all possible rank $5$ fusions are covered in Theorems \ref{Completegraphs}--\ref{Clebsh}. Hence, we will not consider the possibilities for any rank $5$ fusions in our remaining Observations. This result is stated below;

\begin{Lemma}\label{rank5}
The generalized Hamming scheme does not have any non-trivial fusion of rank $5$, other than the special case fusions listed in Theorems \ref{Completegraphs} and \ref{News}.
\end{Lemma}

\begin{proof}
The following table includes all possible rank $5$ fusions of $H(2,\mathcal{A})$. For a fusion to be of rank $5$, we need exactly three isolated matrices. This clearly puts many restrictions on our eigenvalues and which pairs of irreducible characters will be equal. We refer to the fusions arising from an imprimitive scheme in Theorem \ref{Completegraphs} as the imprimitive case. A dash in the final column indicates that the fusion is never possible.

\begin{table}[ht]
  \centering
    \begin{tabular}{||c c c||}
    \hline
 Rank $5$ fusions& Observation
 & Theorem \\ [0.3ex] 
 \hline\hline
 $\{C_{5}, \ C_{9},\ C_{2}+C_{4}, \ C_{3}+C_{7}+C_{6}+C_{8}\}$ & Imprimitive case & \ref{Completegraphs}(i) \\ 
 \hline
 $\{C_{5}, \ C_{9}, \ C_{3}+C_{7}, \ C_{2}+C_{4}+C_{6}+C_{8}\}$ & Imprimitive case & \ref{Completegraphs}(ii)  \\ 
 \hline
$\{C_{5}, \ C_{9}, \ C_{6}+C_{8}, \ C_{2}+C_{4}+C_{3}+C_{7}\}$ & 1,2 & -  \\ 
 \hline
  $\{C_{5}, \ C_{2}+C_{4}, \ C_{3}+C_{7}, \ C_{9}+C_{6}+C_{8}\}$ & 1,3 & -  \\
 \hline
  $\{C_{5}, \ C_{2}+C_{4}, \ C_{6}+C_{8}, \ C_{9}+C_{3}+C_{7}\}$ & 1,3 & - \\
 \hline
$\{C_{5}, \ C_{3}+C_{7}, \ C_{6}+C_{8}, \ C_{9}+C_{2}+C_{4}\}$ & 1,4 & \ref{News}(i)  \\
 \hline
  $\{C_{9}, \ C_{2}+C_{4}, \ C_{6}+C_{8}, \ C_{5}+C_{3}+C_{7}\}$ & 2,3 & \ref{News}(ii)  \\
 \hline
 $\{C_{9}, \ C_{3}+C_{7}, \ C_{6}+C_{8}, 
  C_{5}+C_{2}+C_{4}\}$ & 2,4 & -  \\
 \hline
  $\{C_{9}, \ C_{3}+C_{7}, \ C_{2}+C_{4}, \ C_{5}+C_{6}+C_{8}\}$ & 3,4 & -  \\
 \hline
  $\{C_{2}+C_{4}, \ C_{3}+C_{7}, \ C_{6}+C_{8}, \ C_{5}+C_{9}\}$ & 3,4 & -  \\
 \hline
    \end{tabular}
  \caption{Rank 5 fusions}
\end{table}

\end{proof}

From now on, we will replace the notation $\chi_{i}(C_{j})$ with $\chi_{i}$ as the column or sum of columns that is referred to is obvious from the Observation statement.

 \begin{obs}
    There are no possible rank $3$ or rank $4$, $(C_{2}+C_{4})+C_{5}$-isolating fusions.
    
\end{obs}
    
\begin{proof}
  We must have the following column by itself 
in the new character table for any feasible $(C_{2}+C_{4})+C_{5}$-isolating fusion.
    
  \[
\begin{blockarray}{cc}
\mathbf{(\textit{C}_{2}+\textit{C}_{4})+\textit{C}_{5}} \\
\begin{block}{(c)c}
  2k+k^2 & \chi_{1} \\
  kr+k+r& \chi_{2} \\
  ks+k+s & \chi_{3} \\
  r^2+2r & \chi_{4} \\
  rs+r+s & \chi_{5} \\
  s^2+2s & \chi_{6} \\
\end{block}
\end{blockarray}
 \]
The poset for this column is

 
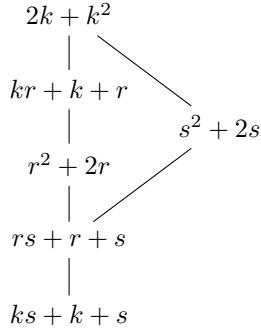
\begin{center}
         \begin{tikzpicture}[scale=.5]
  \node (one) at (0,8) {$2 k +k ^{2}$};
  
  \node (b) at (0,6) {$kr+k+r$};
  
  \node (d) at (0,4) {$r^{2}+2r$};
  \node (c) at (0,2) {$rs+r+s$};
  \node (1) at (0,0) {$ks+k+s$};
  \node (zero) at (4,5) {$s^2+2s$};
  \draw (c) -- (1) (zero) -- (c) -- (d) -- (b) -- (one) -- (zero);
\end{tikzpicture}
    \end{center}
   
  As clear from the height of the above poset, the possible equalities are either $\chi_{4}=\chi_{6}$ when $r=-2-s$, or  $\chi_{2}=\chi_{6}$ when $kr+k+r=s^{2}+2s$. 
  
  There are no possible rank $3$ or rank $4$ $(\textit{C}_{2}+\textit{C}_{4})+\textit{C}_{5}$-isolating fusions.
\end{proof}
\medskip    

\begin{obs}
    A rank $4$, $(C_{3}+C_{7})+C_{5}$-isolating fusion is possible if and only if $k+r=3-s$. This condition is equivalent to $\chi_{2}(C_{3}+C_{7}+C_{5})=\chi_{6}(C_{3}+C_{7}+C_{5})$ and $\chi_{4}(C_{3}+C_{7}+C_{5})=\chi_{3}(C_{3}+C_{7}+C_{5})$. The only possible rank $3$ $(C_{3}+C_{7})+C_{5}$-isolating fusion is covered under Theorem \ref{Clebsch}.
    
\end{obs}
\begin{proof}
  We must have the following sum of columns by itself 
in the new character table for any feasible $(C_{3}+C_{7})+C_{5}$-isolating fusion.   
  \[
\begin{blockarray}{cc}
\mathbf{(\textit{C}_{3}+\textit{C}_{7})+\textit{C}_{5}} \\
\begin{block}{(c)c}
  2\ell+k^2 & \chi_{1} \\
  kr+\ell-r-1 & \chi_{2} \\
  ks+\ell-s-1 & \chi_{3} \\
  r^2-2-2r & \chi_{4} \\
  rs-2-r-s & \chi_{5} \\
  s^2-2-2s & \chi_{6} \\
\end{block}
\end{blockarray}
 \]
  
 The poset for this column is 
  
   
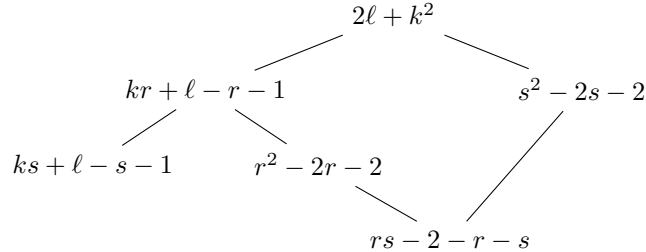
\begin{center}
         \begin{tikzpicture}[scale=.5]
  \node (one) at (0,8) {$2 \ell+k^2$};
  
  \node (b) at (-5,6) {$kr+ \ell -r-1$};
  
  \node (d) at (5,6) {$ s^2-2s-2$};
  \node (c) at (-2,4) {$r^2-2r-2$};
   \node (1) at (-8,4) {$ks+ \ell -s-1$};
  \node (zero) at (1.5,2) {$rs-2-r-s$};
  \draw  (one)--(b)(one)--(d)(b)--(c)(b)--(1)(c)--(zero)(d)--(zero);
\end{tikzpicture}
    \end{center}
   
  We consider all possible rank $4$ fusions from the above poset. We will need exactly four distinct entries in this column to have a rank $4$ fusion.  We end up with four possible cases;
  \begin{enumerate}
     
      \item If $\chi_{2}=\chi_{6}$ and $\chi_{4}=\chi_{3}$ then $\ell+1=r^{2}-2r+s-ks$ and $\ell+1=s^{2}-2s+r-kr$. We combined both the conditions which leads to $k+r=3-s$.
      \item  Assume that $\chi_{2}=\chi_{6}$ and $\chi_{5}=\chi_{3}$. This happens exactly when both the conditions $\ell+1=s^{2}-2s+r-kr$ and $\ell+1=rs-r-ks$ hold. Reducing the conditions gives $k=2-s$ and $\ell=s^2-2s+rs-r-1$.
Putting these into the second and third row orthogonality conditions gives $s(s-1)(r-1)(r+s-2)=0$.
So the two cases that could give feasible fusions are $r=1$ and $r=2-s=k$. The latter case is the imprimitive association scheme and is covered under Theorem \ref{Completegraphs}. For the case when $r=1$, we used GAP to check for possible fusions on the character table and did not get any rank $4$, $(C_{3}+C_{7})+C_{5}$-isolating fusions. 
 \item Assume that $\chi_{4}=\chi_{6}$ and $\chi_{3}=\chi_{5}$. This happens exactly when $r=2-s$ and $\ell+1=rs-r-ks$. We combined both the conditions which leads to $\ell+3+s^{2}=s(3-k)$. This case is not compatible with Observations $2$ or $3$ and hence cannot give a rank $4$ fusion.
 \item The last possible case is when $\chi_{3}=\chi_{4}=\chi_{6}$. This case is also not compatible with Observations $2$ or $3$ in our list. Hence, this case cannot give a rank $4$ fusion. 
  \end{enumerate}
  Hence, only the first and second conditions can possibly lead to a rank $4$ fusion which is covered under Theorem \ref{Crazyrank4}. 
  
  Next, we consider all possible rank $3$ $(\textit{C}_{3}+\textit{C}_{7})+\textit{C}_{5}$-isolating fusions. From the above poset, we can only have one possibility. That is, if $\chi_{2}=\chi_{6}$ and $\chi_{3}=\chi_{4}=\chi_{5}$. This happens exactly when $r=1$ and $\ell+1=s^{2}-2s+r-kr$. We combined both the conditions which leads to $\ell=s^{2}-2s-k$. This fusion is only possible for the case when $r=1$. Hence, the only possible rank $3$, $(\textit{C}_{3}+\textit{C}_{7})+\textit{C}_{5}$-isolating fusion is covered in Theorem \ref{Clebsch} $(ii)$.

\end{proof}
  
\begin{obs}
    There is no possible rank $3$ or rank $4$ $C_{9}+(C_{3}+C_{7})$-isolating fusion.
    
\end{obs}

 \begin{proof}   
    We must have the following column by itself 
in the new character table for any feasible $C_{9}+(C_{3}+C_{7})$-isolating fusion.
  \[
\begin{blockarray}{cc}
\mathbf{\textit{C}_{9}+(\textit{C}_{3}+\textit{C}_{7})} \\
\begin{block}{(c)c}
  2\ell+\ell^2 & \chi_{1} \\
  -r\ell-r-1 & \chi_{2} \\
  -s\ell-s-1 & \chi_{3} \\
  r^2-1 & \chi_{4} \\
  rs-1 & \chi_{5} \\
  s^2-1 & \chi_{6} \\
\end{block}
\end{blockarray}
 \]
 The poset for this column is
     
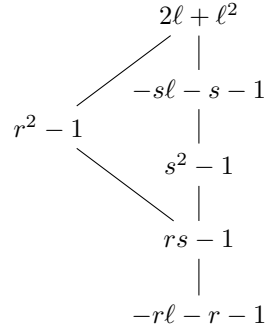
\begin{center}
         \begin{tikzpicture}[scale=.5]
  \node (one) at (0,8) {$2 \ell +\ell ^{2}$};
  
  \node (b) at (0,6) {$-s\ell-s-1$};
  
  \node (d) at (0,4) {$s^{2}-1$};
  \node (c) at (0,2) {$rs-1$};
  \node (1) at (0,0) {$-r\ell-r-1$};
  \node (zero) at (-4,5) {$r^{2}-1$};
  \draw (c) -- (1) (zero) -- (c) -- (d) -- (b) -- (one)--(zero);
\end{tikzpicture}
    \end{center}
 Clearly from the height of the poset, there are no possible $\textit{C}_{9}+(\textit{C}_{3}+\textit{C}_{7})$-isolating fusions of rank $3$ or $4$.

\end{proof}

\begin{Lemma}\label{rank4}
The generalized Hamming scheme does not have any non-trivial fusion of rank $4$ other than the special cases mentioned in Theorems \ref{Completegraphs} and \ref{Payley}.
\end{Lemma}

\begin{proof}
We can divide all possible rank $4$ fusions of $H(2,\mathcal{A})$ into two categories. The first type where we have one isolated matrix and the second is where we have two isolated matrices. They are listed in the following tables respectively.

\begin{table}[ht]
\centering
 \begin{tabular}{||c c c||}
 \hline
 Rank $4$ fusions
 (One isolating matrix)& Observation & Theorem \\ [0.3ex] 
 \hline\hline
   $\{C_{5}, \ C_{9}+(C_{3}+C_{7}), \ (C_{6}+C_{8})+(C_{2}+C_{4})\}$ & 1 & - \\
 \hline
  $\{C_{5}, \ C_{9}+(C_{6}+C_{8}), \ (C_{3}+C_{7})+(C_{2}+C_{4})\}$ & 1 & - \\
 \hline
  $\{C_{5}, \ C_{9}+(C_{2}+C_{4}), \ (C_{6}+C_{8})+(C_{3}+C_{7})\}$ & 1 & - \\
 \hline
    $\{C_{9}, \ C_{5}+(C_{2}+C_{4}), \ (C_{6}+C_{8})+(C_{3}+C_{7})\}$ & 2 & - \\
 \hline
$\{C_{9}, \ C_{5}+(C_{6}+C_{8}), \ (C_{3}+C_{7})+(C_{2}+C_{4})\}$ & 2 & - \\
 \hline
 $\{C_{9}, \ C_{5}+(C_{3}+C_{7}), \ (C_{6}+C_{8})+(C_{2}+C_{4})\}$ & 2 & - \\
 \hline
  $\{(C_{2}+C_{4}), \ C_{5}+(C_{6}+C_{8}), \ C_{9}+(C_{3}+C_{7}),\}$ & 3 & - \\
 \hline
 $\{(C_{2}+C_{4}), \ C_{5}+C_{9}, \ (C_{3}+C_{7})+(C_{6}+C_{8})\}$ & 3 & - \\
 \hline
$\{(C_{2}+C_{4}), \ C_{9}+(C_{6}+C_{8}), \ C_{5}+(C_{3}+C_{7})\}$ & 3 & - \\
 \hline
  $\{(C_{3}+C_{7}), \ (C_{6}+C_{8})+(C_{2}+C_{4}), \ C_{5}+C_{9}\}$ & 4 & - \\
 \hline
  $\{(C_{3}+C_{7}), \ (C_{6}+C_{8})+C_{9}, \ (C_{2}+C_{4})+C_{5}\}$ & 4 & - \\
 \hline
  $\{(C_{3}+C_{7}), \ (C_{6}+C_{8})+C_{5}, \ (C_{2}+C_{4})+C_{9}\}$ & 4 & - \\
 \hline
 $\{C_{6}+C_{8}, \ C_{5}+C_{9}, \ (C_{2}+C_{4})+(C_{3}+C_{7})\}$ & - & \ref{Payley}  \\ 
 \hline
  $\{C_{6}+C_{8}, \ C_{5}+(C_{2}+C_{4}), \ C_{9}+(C_{3}+C_{7})\}$ & 5,\ 7 & -  \\ 
 \hline
   $\{C_{6}+C_{8}, \ C_{9}+(C_{2}+C_{4}), \ C_{5}+(C_{3}+C_{7})\}$ & 6 & \ref{Crazyrank4} \\ 
 \hline
\end{tabular}
\caption{Rank 4 fusions with one isolating matrix}
\end{table}

\begin{table}[H]
\centering
 \begin{tabular}[t]{||c c c||}
 \hline
 Rank $4$ fusions
 (Two isolating matrices)& Observation & Theorem \\ [0.3ex] 
 \hline\hline
 $\{C_{5}, \ C_{9}, \ C_{2}+C_{4}+C_{3}+C_{7}+C_{6}+C_{8}\}$ & 1 & -  \\ 
 \hline
  $\{C_{5}, \ C_{2}+C_{4}, \ C_{9}+C_{3}+C_{7}+C_{6}+C_{8}\}$ & Imprimitive case & \ref{Completegraphs}(i)  \\ 
 \hline
 $\{C_{5}, \ C_{3}+C_{7}, \ C_{2}+C_{4}, \ C_{9}+C_{6}+C_{8}\}$  & 1 & -  \\
 \hline
   $\{C_{5}, \ C_{6}+C_{8}, \ C_{2}+C_{4}+C_{9}+C_{3}+C_{7}\}$ & 1 & -  \\
 \hline
  $\{C_{9}, \ C_{2}+C_{4}, \ C_{5}+C_{3}+C_{7}+C_{6}+C_{8}\}$ & 2 & - \\
 \hline
   $\{(C_{9}, \ C_{3}+C_{7}, \ C_{5}+C_{2}+C_{4}+C_{6}+C_{8}\}$ & Imprimitive Case & \ref{Completegraphs}(ii) \\
 \hline
 $\{(C_{9}, \ C_{6}+C_{8}, \ C_{5}+C_{2}+C_{4}+C_{3}+C_{7}\}$  & 2 & - \\
 \hline
$\{C_{2}+C_{4}, \ C_{3}+C_{7}, \ C_{9}+
 C_{5}+C_{2}+C_{4}\}$  & 3 & - \\
 \hline
 $\{C_{2}+C_{4}, \ C_{6}+C_{8}, \ C_{3}+C_{7}+C_{9}+C_{5}\}$ & 3 & - \\
 \hline
 $\{C_{3}+C_{7}, \ C_{6}+C_{8}, \ C_{2}+C_{4}+C_{9}+C_{5}\}$ & 4 & - \\
 \hline
\end{tabular}
\caption{Rank 4 fusions with two isolating matrices}
\end{table}

We used Observations $1-7$ for ruling out possible rank $4$ fusions. Hence, the only possible rank $4$ fusions of the generalized Hamming scheme are the ones coming from special cases.

\end{proof}


Finally, we start by observing the conditions on eigenvalues for possible rank $3$ fusions.

\medskip
\begin{obs}

     A rank $3$ or rank $4$, $(C_{5}+C_{9})$-isolating fusion is possible if and only if $k=\ell$ and $r=-1-s$. This condition is equivalent to $\chi_{2}(C_{5}+C_{9})=\chi_{4}(C_{5}+C_{9})=\chi_{6}(C_{5}+C_{9})$ and $\chi_{3}(C_{5}+C_{9})=\chi_{5}(C_{5}+C_{9})$. The only possible rank $3$ and rank $4$ $(C_{5}+C_{9})$-isolating fusions are covered under Theorem \ref{Payley}.
    
\end{obs}
\begin{proof}
    We must have the following sum of columns by itself 
in the new character table for any feasible $(C_{5}+C_{9})$-isolating fusion.
    
  \[
\begin{blockarray}{cc}
\mathbf{\textit{C}_{5}+\textit{C}_{9}} \\
\begin{block}{(c)c}
  \ell^2+k^2 & \chi_{1} \\
  kr-\ell(r+1) & \chi_{2} \\
  ks-\ell(s+1) & \chi_{3} \\
  r^2+(1+r)^{2} & \chi_{4} \\
  2rs+r+s+1 & \chi_{5} \\
  s^2+(1+s)^{2}& \chi_{6} \\
\end{block}
\end{blockarray}
 \]
  
  The poset for this column is
  
   
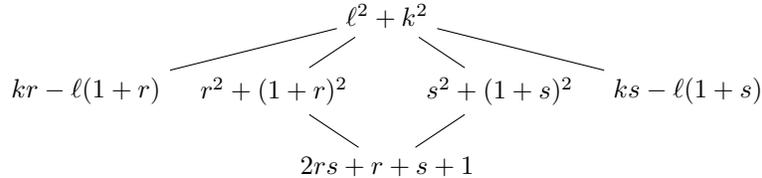
\begin{center}
         \begin{tikzpicture}[scale=.5]
  \node (one) at (0,8) {$ \ell^2+k^2$};
  
  \node (b) at (0,4) {$2rs+r+s+1$};
  \node (c) at (-8,6) {$kr-\ell (1+r)$};
   \node (f) at (8,6) {$ks-\ell (1+s)$};
  \node (d) at (-3,6) {$r^2+(1+r)^2 $};
   \node (e) at (3,6) {$s^2+(1+s)^2 $};
  \draw       (one)-- (d) (one) --(e) (b)--(d) (b) --(e) (one)--(f) (one) -- (c)    ;
\end{tikzpicture}
    \end{center}

   From the poset diagram, it is clear that a rank $3$ fusion is only possible if $r=-1-s$ implying $\chi_{4}(C_{5}+C_{9})=\chi_{6}(C_{5}+C_{9})$. We can have the following four cases:
   \begin{enumerate}
     \item Assume that $\chi_{2}=\chi_{4}=\chi_{6}$ and $\chi_{3}=\chi_{5}$. This happens exactly when $ks+\ell r=2rs=-2r-2r^2$ and $kr+\ell s=1+2s+2s^2$. Adding both equations, we get $k+\ell=-1$ which gives us a contradiction. 
      \item If $\chi_{2}=\chi_{3}=\chi_{4}=\chi_{6}$ then $k=\ell=-(2r^{2}+2r+1)$, which is a contradiction since $k>0$.
   \item Assume that $\chi_2=\chi_3=\chi_5$ and  $\chi_4=\chi_6$. This happens exactly when the conditions $r=-1-s$ and $k=\ell$ hold simultaneously. Both these conditions imply $k=\ell=2(r+r^2)$.  Hence, we get a rank $3$ fusion covered in Theorem \ref{Payley}. 
     
       \item If $\chi_{2}=\chi_{5}$ and $\chi_{3}=\chi_{4}=\chi_{6}$ then $kr+\ell s=2rs=-2r-2r^2$ and $ks+\ell r=1+2r+2r^2$. Adding both equations, we get $k+\ell=(r+s)^{-1}$ which is a contradiction.

   \end{enumerate}

 Hence, we can conclude that the only rank $3$ fusion for this case is covered by Theorem \ref{Payley}. This Observation combined together with two later Observations also gives a rank $4$ fusion listed in Theorem \ref{Payley}.
  
 \end{proof}

\medskip

\medskip    
 \begin{obs}
 There is no possible rank $3$ $(C_{6}+C_{8})+C_{5}$-isolating fusion of $H(2,\mathcal{A})$.
 \end{obs}   

\begin{proof}
   We must have the following sum of two columns by itself 
in the new character table for any feasible $(C_{6}+C_{8})+C_{5}$-isolating fusion.
    
  \[
\begin{blockarray}{cc}
\mathbf{(\textit{C}_{6}+\textit{C}_{8})+\textit{C}_{5}} \\
\begin{block}{(c)c}
  2k\ell+k^2 & \chi_{1} \\
  -r\ell-k & \chi_{2} \\
  -s\ell-k & \chi_{3} \\
  -r^2-2r& \chi_{4} \\
  -rs-r-s & \chi_{5} \\
  -s^2-2s & \chi_{6} \\
\end{block}
\end{blockarray}
 \]
 
 The poset for this column is
   \begin{center}
         \begin{tikzpicture}[scale=.5]
  \node (one) at (0,6) {$2k \ell +k^2$};
  \node (b) at (0,4) {$-rs-r-s$};
   \node (a) at (5,4) {$-s \ell-k$};
    \node (c) at (5,2) {$-r \ell-k$};
     \node (d) at (2,2) {$-s^2-2s$};
  \node (zero) at (-2,2) {$-r^2-2r$};
  \draw  (one) -- (b) -- (d) (b) -- (zero) (one)--  (a)-- (c);
\end{tikzpicture}
    \end{center}
 As clear from the above lattice, there is only one possible case for a possible rank $3$ fusion.
  \begin{enumerate}
      
       \item If $\chi_{5}=\chi_{3}$ and $\chi_{2}=\chi_{4}=\chi_{6}$ then we must have $r^{2}-2r=r\ell+k$ which implies $r^{2}=r\ell+k+2r$. This is a contradiction since $\ell>1+r$. Hence, this case cannot give a rank $3$ fusion.
  \end{enumerate}
      
There are no possible rank $3$ fusions in this case.

 \end{proof}

\medskip

 \medskip
 
 \begin{obs}

     A rank $3$ $C_{9}+(C_{2}+C_{4})$-isolating fusion is possible if and only if $s=-2$, $k=r(3+r)$ and $\ell=3+r$. This condition is equivalent to $\chi_{2}=\chi_{5}=\chi_{6}$ and $\chi_{3}=\chi_{4}$. The only possible rank $3$ $C_{9}+(C_{2}+C_{4})$-isolating fusion is covered under Theorem \ref{Clebsch}.

 \end{obs}
 
\begin{proof}
We must have the following sum of columns by itself 
in the new character table for any feasible $C_{9}+(C_{2}+C_{4})$-isolating fusion.
 \[
\begin{blockarray}{cc}
\mathbf{\textit{C}_{9}+(\textit{C}_{2}+\textit{C}_{4})} \\
\begin{block}{(c)c}
  \ell^2+2k & \chi_{1} \\
  k+r-\ell-r\ell & \chi_{2} \\
  k+s-\ell-s\ell  & \chi_{3} \\
  r^2+4r+1& \chi_{4} \\
  rs+2r+2s+1 & \chi_{5} \\
  s^2+4s+1 & \chi_{6} \\
\end{block}
\end{blockarray}
 \]
 
 The poset for this column is
  
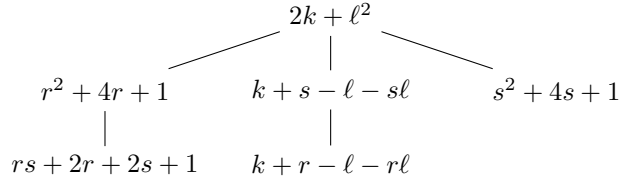
\begin{center}
         \begin{tikzpicture}[scale=.5]
  \node (one) at (0,6) {$2k +\ell^2$};
  \node (b) at (0,4) {$k+s-\ell-s\ell$};
   \node (a) at (0,2) {$k+r-\ell-r\ell$};
    \node (c) at (-6,4) {$r^2+4r+1$};
     \node (d) at (6,4) {$s^2+4s+1$};
  \node (zero) at (-6,2) {$rs+2r+2s+1$};
  \draw  (one) -- (b) -- (a) (c) -- (zero) (one) -- (c) (one) --(d) ;
\end{tikzpicture}
    \end{center}
 From the above poset, if a rank $3$ fusion is possible then one the following cases must hold:
 \begin{enumerate}
     \item  The first possible case is when $\chi_{4}=\chi_{6}=\chi_{3}$ and $\chi_{2}=\chi_{5}$. This happens if and only if $k-\ell=1+3r+r\ell+r^2$ and $k-\ell=1+2r+rs+s+s\ell$. We combined these two equalities and obtained $r+r^2+r\ell=rs+s+s\ell
     $, which gives us a sign contradiction. Hence, this case cannot give a rank $3$ fusion.

      \item  If $\chi_{2}=\chi_{6}=\chi_{5}$ and $\chi_{3}=\chi_{4}$ then $\chi_{5}=\chi_{6}$ implies $s=-2$ and $\chi_{3}=\chi_{4}$ gives $k-\ell=r^2+4r+3-2 \ell$. We combined these two equalities with $\chi_{2}=\chi_{6}$ and obtained $-3-r-r\ell=r^2+4r+3-2\ell$, which gives $\ell=3+r$.
      
 \end{enumerate}

The conditions obtained in (b) gives a rank $3$ fusion when $\chi_{2}=\chi_{6}=\chi_{5}$ and $\chi_{3}=\chi_{4}$ implies $\ell=3+r$, $s=-2$ and $k=3r+r^2$. This fusion is covered under Theorem \ref{Clebsch}. 

We have ruled out possible rank $4$ fusions in Observations $1-7$, this case along with Observation $6$ leads to a rank $4$ fusion covered under Theorem \ref{Crazyrank4}.
 \end{proof}
 \medskip

 \begin{obs}
 There are no possible rank $3$ $C_{9}+(C_{6}+C_{8})$-isolating fusions of $H(2,\mathcal{A})$.
 \end{obs}

\begin{proof} 

 We must have the following sum of columns by itself 
in the new character table for any feasible $C_{9}+(C_{6}+C_{8})$-isolating fusion.
  \[
\begin{blockarray}{cc}
\mathbf{\textit{C}_{9}+(\textit{C}_{6}+\textit{C}_{8})} \\
\begin{block}{(c)c}
  2k\ell+\ell^2 & \chi_{1} \\
  -rk-k-\ell & \chi_{2} \\
  -sk-k-\ell & \chi_{3} \\
  -r^2+1 & \chi_{4} \\
  -rs+1 & \chi_{5} \\
  -s^2+1 & \chi_{6} \\
\end{block}
\end{blockarray}
 \]
 The poset for this column is
    
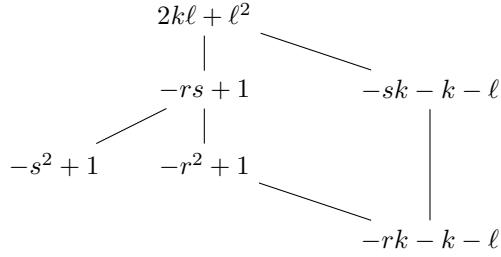
\begin{center}
         \begin{tikzpicture}[scale=.5]
  \node (one) at (0,8) {$2k \ell +\ell ^{2}$};
  
  \node (b) at (6,6) {$-sk-k-\ell$};
  
  \node (d) at (6,2) {$-rk-k-\ell$};
  \node (c) at (0,6) {$-rs+1$};
  \node (1) at (-4,4) {$-s^2+1$};
  \node (zero) at (0,4) {$-r^2+1$};
  \draw (c)--(1) (zero) -- (c)--(one)-- (b)--(d) (zero)--(d);
\end{tikzpicture}
    \end{center}
As clear from the poset above. There are no possible rank $3$ fusions for this case.
 \end{proof}

\medskip

 \begin{obs}
 A rank $3$ $(C_{6}+C_{8})$-isolating fusion is possible if and only if $k=\ell$ and $r=-1-s$. This condition is equivalent to $\chi_{2}=\chi_{3}=\chi_{4}=\chi_{6}$. The only possible rank $3$ $(C_{6}+C_{8})$-isolating fusion is covered under Theorem \ref{Payley}. The two possible rank $4$ $(C_{6}+C_{8})$-isolating fusions are covered in Theorem \ref{Crazyrank4} and \ref{Payley}.
 \end{obs}

\begin{proof}
 We must have the following column by itself 
in the new character table for any feasible $(C_{6}+C_{8})$-isolating fusion.
   \[
\begin{blockarray}{cc}
\mathbf{(\textit{C}_{6}+\textit{C}_{8})} \\
\begin{block}{(c)c}
  2k\ell & \chi_{1} \\
  -k-kr+r\ell & \chi_{2} \\
  -k-ks+s\ell & \chi_{3} \\
  -2r-2r^2 & \chi_{4} \\
  -2rs-r-s & \chi_{5} \\
  -2s-2s^2 & \chi_{6} \\
\end{block}
\end{blockarray}
 \]
 The poset for this column is
   
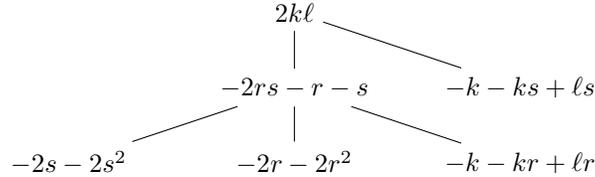
\begin{center}
         \begin{tikzpicture}[scale=.5]
  \node (one) at (0,2) {$2 k \ell$};
  
  \node (b) at (0,0) {$-2rs-r-s$};
  
  \node (d) at (6,-2) {$-k-kr+\ell r$};
  \node (c) at (6,0) {$-k-ks+\ell s$};
  \node (1) at (0,-2) {$-2r-2r^2$};
  \node (zero) at (-6,-2) {$-2s-2s^2$};
  \draw (one) -- (b) --(d) (one) -- (c)  (b) -- (1) (b) -- (zero);
\end{tikzpicture}
    \end{center}

   From the above poset, it is clear that a rank $3$ fusion is only possible if $\chi_{2}=\chi_{4}=\chi_{6}$ implying $r=-1-s$. We can have the following cases:
   \begin{enumerate}
   
   \item If $\chi_2=\chi_3=\chi_4=\chi_6$ then $r=-1-s$ and $k=\ell$ hold simultaneously implying $k=\ell=2(r+r^2)$. Hence, we get a rank $3$ fusion covered in Theorem \ref{Payley}. 
       \item If $\chi_{2}=\chi_{4}=\chi_{6}$ and $\chi_{3}=\chi_{5}$ then $ks+\ell r=2rs=-2r-2r^2$ and $kr+\ell s=1+2s+2s^2$ respectively. Adding both equations, we get $k-\ell=1+2r$. This fusion is a special case listed under Theorem \ref{Payley}. 
      
   \end{enumerate}

 Hence, we conclude that the only rank $3$ fusion for this case is covered by Theorem \ref{Payley}. There are no other rank $3$ fusions possible. We have covered all possible rank $4$ fusions arising from this case using previous Observations in Lemma \ref{rank4}.
  
\end{proof}
  \medskip
 
\begin{obs}
 A rank $3$ $(C_{2}+C_{4})+(C_{6}+C_{8})$-isolating fusion is possible if and only if $k-\ell=1+2r$ and $r=-s$. This condition is equivalent to $\chi_{2}=\chi_{4}=\chi_{6}$ and $\chi_{3}=\chi_{5}$. The only possible rank $3$ $(C_{2}+C_{4})+(C_{6}+C_{8})$-isolating fusion is covered under part (i) of Theorem \ref{Clebsh}. There are no rank $4$ fusions for this case.
\end{obs}

\begin{proof}  

 We must have the following sum of columns by itself 
in the new character table for any feasible $(C_{2}+C_{4})+(C_{6}+C_{8})$-isolating fusion.
  \[
\begin{blockarray}{cc}
\mathbf{(\textit{C}_{2}+\textit{C}_{4})+(\textit{C}_{6}+\textit{C}_{8})} \\
\begin{block}{(c)c}
  2k\ell+2k & \chi_{1} \\
  -r(k-\ell-1) & \chi_{2} \\
  -s(k-\ell-1) & \chi_{3} \\
  -2r^2 & \chi_{4} \\
  -2rs & \chi_{5} \\
  -2s^2& \chi_{6} \\
\end{block}
\end{blockarray}
 \]
 
 The poset for this column is
   
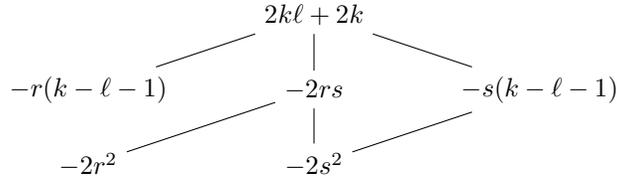
\begin{center}
         \begin{tikzpicture}[scale=.5]
  \node (one) at (0,8) {$2k \ell +2k$};
  
  \node (b) at (-6,6) {$-r(k-\ell-1)$};
  
  \node (d) at (6,6) {$-s(k-\ell-1)$};
  \node (c) at (0,6) {$-2rs$};
  \node (1) at (0,4) {$-2s^2$};
  \node (zero) at (-6,4) {$-2r^2$};
  \draw (one)--(b)(one)--(c)(one)--(d) (c)--(1)(c)--(zero)(d)--(1);
\end{tikzpicture}
    \end{center}

 From the above poset, we can conclude that the only possible rank $3$ fusions for this case are:
 \begin{enumerate}
     \item If $\chi_{2}=\chi_{3}=\chi_{5}$ and $\chi_{4}=\chi_{6}$ then the first equality implies $k-\ell=1$. This is further reduced to $-2rs=0$ which is a contradiction since $s<-1$ and $r>0$. Hence, this case cannot give us a rank $3$ fusion.
     \item If $\chi_{2}=\chi_{4}=\chi_{6}$ and $\chi_{3}=\chi_{5}$ then both the conditions $r=-s$ and $k-\ell=1+2r$ hold. In this case, we get a rank $3$ fusion covered under Theorem \ref{Clebsh}.
 \end{enumerate}
 
 Hence, the only possible rank $3$ fusion from this case is covered under the Theorem \ref{Clebsh}.
 
 \end{proof}

\medskip

\begin{obs}
  A rank $3$ $(C_{3}+C_{7})+(C_{6}+C_{8})$-isolating fusion is possible if and only if $\ell-k-1=-2(r+1)$ and $r=-2-s$. This condition is equivalent to $\chi_{2}=\chi_{4}=\chi_{6}$ and $\chi_{3}=\chi_{5}$. The only possible rank $3$ $(C_{3}+C_{7})+(C_{6}+C_{8})$-isolating fusion is covered under part (ii) of Theorem \ref{Clebsh}. There are no rank $4$ fusions for this case.
\end{obs}

\begin{proof}
     We must have the following sum of columns by itself 
in the new character table for any feasible $(C_{3}+C_{7})+(C_{6}+C_{8})$-isolating fusion.
  \[
\begin{blockarray}{cc}
\mathbf{(\textit{C}_{3}+\textit{C}_{7})+(\textit{C}_{6}+\textit{C}_{8})} \\
\begin{block}{(c)c}
  2k\ell+2\ell & \chi_{1} \\
  (r+1)(-k+\ell-1) & \chi_{2} \\
  (s+1)(-k+\ell-1) & \chi_{3} \\
  -2(r+1)^2 & \chi_{4} \\
  -2(r+1)(s+1) & \chi_{5} \\
  -2(s+1)^2 & \chi_{6} \\
\end{block}
\end{blockarray}
 \]
 The poset for this column is
    
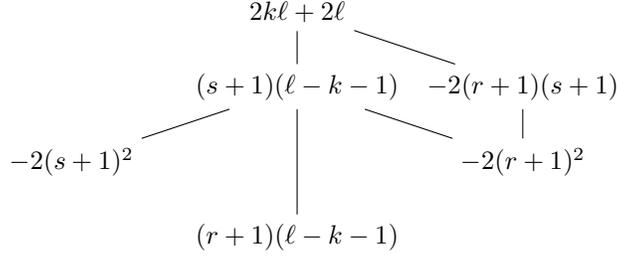
\begin{center}
         \begin{tikzpicture}[scale=.5]
  \node (one) at (0,8) {$2k \ell +2\ell$};
  
  \node (b) at (6,6) {$-2(r+1)(s+1)$};
  
  \node (d) at (6,4) {$-2(r+1)^2$};
  \node (c) at (0,2) {$(r+1)(\ell-k-1)$};
  \node (1) at (-6,4) {$-2(s+1)^2$};
  \node (zero) at (0,6) {$(s+1)(\ell-k-1)$};
  \draw (one)--(zero)--(1)(zero)--(c)(b)--(d)(zero)--(d)(one)--(b);
\end{tikzpicture}
    \end{center}
It is clear from the above poset that a rank $3$ fusion is possible if and only if $\chi_{3}=\chi_{5}$ and $\chi_{2}=\chi_{4}=\chi_{6}$. This is equivalent to  $\ell-k-1=-2(r+1)$ and $r=-2-s$. This rank $3$, $(C_{3}+C_{7})+(C_{6}+C_{8})$-isolating fusion of $H(2,\mathcal{A})$ is covered under part (ii) of Theorem \ref{Clebsh}.
\end{proof}
 \medskip

 \begin{obs}
   The only feasible rank $3$ $(C_{2}+C_{4})+(C_{3}+C_{7})$-isolating fusion of $H(2,\mathcal{A})$ is the homogeneous fusion covered in Corollary \ref{homofusion}. This fusion is independent of any conditions on the eigenvalues.
 \end{obs}
\begin{proof}
    
  We must have the following sum of columns by itself 
in the new character table for any feasible $(C_{2}+C_{4})+(C_{3}+C_{7})$-isolating fusion.  
  \[
\begin{blockarray}{cc}
\mathbf{(\textit{C}_{2}+\textit{C}_{4})+(\textit{C}_{3}+\textit{C}_{7})} \\
\begin{block}{(c)c}
  2k+2\ell & \chi_{1} \\
  k+\ell-1 & \chi_{2} \\
  k+\ell-1 & \chi_{3} \\
  -2 & \chi_{4} \\
  -2 & \chi_{5} \\
  -2 & \chi_{6} \\
\end{block}
\end{blockarray}
 \]
 
 The only possible rank $3$ fusion for this case is the homogeneous fusion discussed in Corollary \ref{homofusion}.
 \end{proof}


\begin{Lemma}\label{rank3}
The generalized Hamming scheme does not have any non-trivial fusions of rank $3$ other than the homogeneous fusion and the special case fusions in Theorems \ref{Completegraphs}--\ref{Clebsh}.
\end{Lemma}

\begin{proof}
We start by listing all possible rank $3$ fusions of $H(2,\mathcal{A})$ along with the Observation showing us why it is not a feasible fusion.
Note that in the case of Observation $15$, there are three distinct rows which do not depend on the eigenvalues. The rank $3$ fusion for this case is the homogeneous fusion defined in Corollary $5.1$ earlier in this section. Other rank $4$ fusions for this case are either characterized or ruled out in our complete list of fusions in Lemma \ref{rank4}.
\begin{table}[ht]
\centering
 \begin{tabular}{||c c c||}
 \hline
 Rank $3$ fusions & Observation & Theorem \\ [0.3ex] 
 \hline\hline
$\{C_{5}, \ C_{9}+(C_{2}+C_{4})+(C_{3}+C_{7})+C_{6}+C_{8}\}$  & 1 & -  \\ 
 \hline
 $\{C_{9}, \ C_{5}+(C_{2}+C_{4})+(C_{3}+C_{7})+C_{6}+C_{8}\}$ & 2 & -  \\ 
 \hline
 $\{C_{2}+C_{4}, \ C_{5}+C_{9}+(C_{3}+C_{7})+C_{6}+C_{8}\}$ & 3 & -  \\ 
 \hline
 $\{C_{3}+C_{7}, \ C_{5}+C_{9}+(C_{2}+C_{4})+C_{6}+C_{8}\}$ & 4 & -  \\ 
 \hline
 $\{(C_{6}+C_{8}), \ C_{5}+C_{9}+(C_{3}+C_{7})+(C_{2}+C_{4}\}$ & 11 & \ref{Payley} \\ 
 \hline
 $\{C_{5}+C_{9}, \ (C_{6}+C_{8})+(C_{2}+C_{4})+(C_{3}+C_{7})\}$ & 8 & \ref{Payley} \\
 \hline
  $\{C_{5}+(C_{2}+C_{4}), \ (C_{6}+C_{8})+(C_{3}+C_{7})+C_{9}\}$ & Imprimitive case & \ref{Completegraphs}(i)  \\
 \hline
 $\{C_{5}+(C_{3}+C_{7}), \ (C_{6}+C_{8})+(C_{2}+C_{4})+C_{9}\}$ & 6 & \ref{Clebsch}(ii) \\
 \hline
  $\{C_{5}+(C_{6}+C_{8}), \ (C_{3}+C_{7})+(C_{2}+C_{4})+C_{9}\}$ & 9 & -  \\
 \hline
  $\{C_{9}+(C_{2}+C_{4}), \ (C_{6}+C_{8})+(C_{3}+C_{7})+C_{5}\}$ & 10 & \ref{Clebsch}(i) \\
 \hline
  $\{C_{9}+(C_{3}+C_{7}), \ (C_{6}+C_{8})+(C_{2}+C_{4})+C_{5}\}$ & Imprimitive case & \ref{Completegraphs}(ii) \\
 \hline
  $\{C_{9}+(C_{6}+C_{8}), \ (C_{3}+C_{7})+(C_{2}+C_{4})+C_{5}\}$ & 10 & - \\
 \hline
  $\{(C_{2}+C_{4}(+(C_{6}+C_{8}), \ (C_{3}+C_{7})+C_{5}+C_{9}\}$ & 13 & \ref{Clebsh}(i)\\
  \hline
  $\{(C_{2}+C_{4})+(C_{3}+C_{7}), \ (C_{6}+C_{8})+C_{5}+C_{9}\}$ & 15 & Cor. $5.1$\\
 \hline
   $\{(C_{3}+C_{7})+(C_{6}+C_{8}), \ (C_{2}+C_{4})+C_{5}+C_{9}\}$ & 14 & \ref{Clebsh}(ii) \\
 \hline
\end{tabular}
\caption{Rank 3 fusions}
\end{table}

Hence, it is shown that the only rank $3$ fusions of $H(2,\mathcal{A})$ are either the homogeneous fusion or the special case fusions in Theorems  \ref{Completegraphs}-\ref{Clebsh}
\end{proof}

The following Observations are the duals of Observations $1-15$. We will refer to these for proving the converse of our main result.

\begin{obs}
  A rank $3$ $C_{5}+C_{9}+(C_{3}+C_{7})$-isolating fusion is possible if and only if $k-\ell=1+2r$ and $r=-s$. This condition is equivalent to $\chi_{2}=\chi_{4}=\chi_{6}$ and $\chi_{3}=\chi_{5}$. The only possible rank $3$ $C_{5}+C_{9}+(C_{3}+C_{7})$-isolating fusion is covered under part (i) of Theorem \ref{Clebsh}. There are no rank $4$ fusions for this case. 
 \end{obs}
\begin{proof}
    
    We must have the following sum of columns by itself 
in the new character table for any feasible $C_{5}+C_{9}+(C_{3}+C_{7})$-isolating fusion.
  \[
\begin{blockarray}{cc}
\mathbf{\textit{C}_{5}+\textit{C}_{9}+(\textit{C}_{3}+\textit{C}_{7})} \\
\begin{block}{(c)c}
  2\ell+\ell^{2}+k^{2} & \chi_{1} \\
  r(k-\ell-1)-1 & \chi_{2} \\
  s(k-\ell-1)-1 & \chi_{3} \\
  -1+2r^{2} & \chi_{4} \\
  -1+2rs & \chi_{5} \\
  -1+2s^{2} & \chi_{6} \\
\end{block}
\end{blockarray}
 \]
 
 As we can note that the entries of this isolating fusion are the dual of those in Observation $13$. It follows a similar proof.
 \end{proof}

 \begin{obs}
 A rank $3$ $(C_{2}+C_{4})+(C_{3}+C_{7})+C_{5}+C_{9}$-isolating fusion is possible if and only if $k=\ell$ and $r=-1-s$. This condition is equivalent to $\chi_{2}=\chi_{3}=\chi_{4}=\chi_{6}$. The only possible rank $3$ $(C_{2}+C_{4})+(C_{3}+C_{7})+C_{5}+C_{9}$-isolating fusion is covered under Theorem \ref{Payley}.
 \end{obs} 

\begin{proof}
We must have the following sum of columns by itself 
in the new character table for any feasible $(C_{2}+C_{4})+(C_{3}+C_{7})+C_{5}+C_{9}$-isolating fusion.
  \[
\begin{blockarray}{cc}
\mathbf{(\textit{C}_{3}+\textit{C}_{7})+(\textit{C}_{2}+\textit{C}_{4})+\textit{C}_{5}+\textit{C}_{9}} \\
\begin{block}{(c)c}
  2\ell+2k+\ell^{2}+k^{2} & \chi_{1} \\
  r(k-\ell)+k-1 & \chi_{2} \\
  s(k-\ell)+k-1 & \chi_{3} \\
  -1+2r^{2}+2r & \chi_{4} \\
  -1+2rs+r+s & \chi_{5} \\
  -1+2s^{2}+2s & \chi_{6} \\
\end{block}
\end{blockarray}
 \]
 
 As we can note that the entries of this isolating fusion are the dual of those in Observation $12$. It follows a similar proof.
\end{proof}

 \begin{obs}

  A rank $3$ $(C_{3}+C_{7})+C_{5}+(C_{6}+C_{8})$-isolating fusion is possible if and only if $s=-2$, $k=r(3+r)$ and $\ell=3+r$. This condition is equivalent to $\chi_{2}=\chi_{5}=\chi_{6}$ and $\chi_{3}=\chi_{4}$. The only possible rank $3$ $(C_{3}+C_{7})+C_{5}+(C_{6}+C_{8})$-isolating fusion is covered under Theorem \ref{Clebsch}.
 \end{obs}

\begin{proof}
We must have the following sum of columns by itself 
in the new character table for any feasible $(C_{3}+C_{7})+C_{5}+(C_{6}+C_{8})$-isolating fusion.
  \[
\begin{blockarray}{cc}
\mathbf{(\textit{C}_{3}+\textit{C}_{7})+\textit{C}_{5}+(\textit{C}_{6}+\textit{C}_{8}}) \\
\begin{block}{(c)c}
  2\ell+2k+k^{2} & \chi_{1} \\
  -k+(1+r)(\ell-1) & \chi_{2} \\
  -k+(1+s)(\ell-1) & \chi_{3} \\
  r^{2}-2-4r & \chi_{4} \\
  -2-2r-2s-rs & \chi_{5} \\
  s^{2}-2-4s & \chi_{6} \\
\end{block}
\end{blockarray}
 \]
 
 As we can note that the entries of this isolating fusion are the dual of those in Observation $10$. It follows a similar proof.
\end{proof}

\begin{obs}

 A rank $3$ $(C_{2}+C_{4})+(C_{3}+C_{7})+(C_{6}+C_{8})$-isolating fusion is possible if and only if $k=\ell$ and $r=-1-s$. This condition is equivalent to $\chi_{2}=\chi_{4}=\chi_{6}$ and $\chi_{3}=\chi_{5}$. The only possible rank $3$ and rank $4$ $(C_{2}+C_{4})+(C_{3}+C_{7})+(C_{6}+C_{8})$-isolating fusions are covered under Theorem \ref{Payley}. 
 \end{obs}

\begin{proof}

We must have the following sum of columns by itself 
in the new character table for any feasible $(C_{2}+C_{4})+(C_{3}+C_{7})+(C_{6}+C_{8})$-isolating fusion.
  \[
\begin{blockarray}{cc}
\mathbf{(\textit{C}_{2}+\textit{C}_{4})+(\textit{C}_{3}+\textit{C}_{7}})+(\textit{C}_{6}+\textit{C}_{8})\\
\begin{block}{(c)c}
  2\ell+2k+2k\ell & \chi_{1} \\
  \ell-1-r(k-\ell) & \chi_{2} \\
  \ell-1-s(k-\ell) & \chi_{3} \\
  -2-2r^{2}-2r & \chi_{4} \\
  -2-2rs-r-s & \chi_{5} \\
  -2-2s^{2}-2s & \chi_{6} \\
\end{block}
\end{blockarray}
 \]
 
 As we can note that the entries of this isolating fusion are the dual of those in Observation $8$. It follows a similar proof.
\end{proof}

\medskip

\begin{theorem}\label{mainresult}
If $\mathcal{A}$ does not belong to the families of strongly-regular graphs in Theorems \ref{Completegraphs}--\ref{Clebsh}, then the only nontrivial fusion of $H(2,\mathcal{A})$ is the homogeneous fusion $H(2,\mathcal{T}_{\mathcal{A}})$.

\end{theorem}

\begin{proof}
As we can see from Lemmas \ref{rank5}--\ref{rank3}, $38$ out of $52$ possible fusions of $H(2,\mathcal{A})$ are not feasible using the Bannai-Muzychuk criterion in Lemma \ref{Bannai-Muzychuk}. In this Section, we prove that all of the $11$ special case fusions of $H(2,\mathcal{A})$ are only feasible if and only if the corresponding conditions on eigenvalues listed in Theorems \ref{Completegraphs}--\ref{Clebsh} are satisfied.

A major part of the proof is written in the list of Observations above. We know that the rank of a fusion corresponds to the number of columns in the character table which are analysed case by case and are referenced by their Observation numbers as they appear in the fusion. The conditions obtained for special case fusions covered in Lemmas \ref{rank5}--\ref{rank3} are also referenced by the Observation number.

\medskip

All that is left to show is that all the feasible nontrivial fusions occur when $\mathcal{A}$ belongs to one of the families in Theorems \ref{Completegraphs}--\ref{Clebsh}. For example, the first four non-trivial fusions (these are the fusions denoted by (1)--(4) in Theorem \ref{Completegraphs}) hold only for the rank $3$ imprimitive scheme and hence are ruled out for any other strongly-regular graphs in Observations $1$--$5$ where we assume that $s<-1$ and $k\neq r$. We prove two examples here and the rest of the other fusions have similar proofs.

\begin{enumerate}
    \item Consider the nontrivial rank $5$ fusion in Theorem \ref{News}(i). The columns of the character table for this fusion are the ones we discussed in Observations $1$, $4$, $10$ and $12$. The corresponding restrictions on eigenvalues obtained are $kr=s^{2}$, $\ell=r-2s-1$ and $-k-kr+r\ell=-2s-2s^{2}$. We express $-k-kr+r\ell=-2s-2s^{2}$ in terms of $r$ and $s$ using the first two equations. After simplification we get a quadratic equation in $s$,
    $$
    s^{2}\left(\frac{1}{r}-1\right)+s(2-2r)+(r^2-r)=0.
    $$
    
    The quadratic equation has no roots when $r>1$. Therefore, we must have $r=1$ for this rank $5$ fusion to be feasible. Hence, we get the same restrictions as in part (i) of Theorem \ref{News}.
    \item Next we consider the nontrivial rank $5$ fusion in part (ii) of Theorem \ref{News}. The columns of the character table for this fusion are the ones we discussed in Observations $2$, $3$, $6$ and $12$. The corresponding restrictions on eigenvalues obtained are $k+s=2r$, $-\ell(1+s)=(1+r)^{2}$ and $-k-ks+s\ell=-2r-2r^{2}$. Hence, we get the same restrictions as in part (ii) of Theorem \ref{News}.
\end{enumerate}

\begin{table}[ht]
\centering
 \begin{tabular}{||c c c||}
 \hline
 Feasible Fusions
 & Observation(s) & Eigenvalue Classification \\ [0.3ex] 
 \hline\hline
  (1) and (1') & 1, 2, 3 & Imprimitive graphs in \ref{Completegraphs} \\
 \hline
  (2) and (2') & 1, 3 & Imprimitive graphs in \ref{Completegraphs} \\
 \hline
   (3) and (3') & 2, 5 & Imprimitive graphs in \ref{Completegraphs} \\
 \hline
   (4) and (4') & 5 & Imprimitive graphs in \ref{Completegraphs} \\
 \hline
 Fusion (5) & 1, 4, 10, 12 & \ref{News}(i) \\
 \hline
Fusion (6) & 8, 12, 15 &  \ref{Payley}\\
 \hline
  Fusion (7)  & 8, 19 & \ref{Payley} \\
 \hline
 Fusion (8) & 12, 17 & \ref{Payley}\\
 \hline
 Fusion (9)  & 6, 10, 12 & \ref{Crazyrank4} \\
 \hline
  Fusion (5')&  2, 3, 6, 12 & \ref{News}(ii) \\
 \hline
  (10) and (10') & 10, 18 & \ref{Clebsch} \\
 \hline
  (11) and (11') & 13, 16 & \ref{Clebsh} \\
 \hline

\end{tabular}
\caption{Feasible fusions}

\end{table}
\medskip

Therefore, it is clear from the above data that the only fusion of the generalized Hamming scheme that is independent of parameters is the homogeneous fusion. Every other fusion is only 
satisfied for special graph families as listed above.

\end{proof}

\section{Future work}

In this section, we illustrate the idea of fusions of $H(2,\mathcal{A})$ using table algebras. As an algebra the adjacency algebra of an association scheme is a standard integral \textit{table algebra}. The \textit{wreath product} of two table algebras $\mathcal{A}=\{A_{0},A_{1},A_{2}\}$ and $\mathcal{B}=\{B_{0},B_{1},B_{2}\}$ is given by $$ \mathcal{A}\wr \mathcal{B}=\{A_{i}\otimes B_{0}; \  i\geq 0  , \  (A_{0}+A_{1}+A_{2}) \otimes B_{j};  \ j\geq 0\}.
$$ The tensor product of two table algebras is a table algebra. Hence, $\mathcal{A}\otimes \mathcal{A}$ is a table algebra. As a table algebra, $H(n,\mathcal{A})$ is a fusion of $\mathcal{A}\otimes \mathcal{A}$ denoted by $\text{Sym}^{n}(\mathcal{A})$ that is obtained by fusing symmetric tensors of basis elements. For more details on table algebras, see Blau~\cite{TA}.

We applied the Bannai-Muzychuk criteria (see Lemma \ref{Bannai-Muzychuk})
on the character table of the rank $9$ tensor product scheme $\mathcal{A}\otimes \mathcal{A}$  to calculate all fusions that are guaranteed regardless of the eigenvalues. We calculated this using GAP~\cite{GAP4}. The GAP code used for this project is open sourced at GitHub repository\footnote{\url{https://github.com/nehamainali/GeneralizedHammingScheme}}. 

The association scheme $\mathcal{A}\otimes\mathcal{A}$ has a natural automorphism
${}': x\otimes y\mapsto y\otimes x$. A partition $\tau$ of the
standard basis for $\{A_i\otimes A_j\}_{i,j=0}^2$ defines a fusion if and only if $\tau'$ determines a fusion. 
If the partitions $\tau$ and $\tau'$ determine isomorphic fusions, then the resulting isomorphism between the fusions is denoted by $\cong$.
The list of all fusions that exist for any rank $3$ association scheme $\mathcal{A}$ are given below.

\[
\begin{array}{rcl}
    \mathcal{A} \otimes \mathcal{A} &=& \{C_{1}, \ C_{2}, \ C_{3}, \ C_{4}, \ C_{5}, \ C_{6}, \ C_{7}, \ C_{8}, 
     \ C_{9}\} \\
     \mathcal{T}_{\mathcal{A}\otimes \mathcal{A}} &=&\{C_{1}, \ C_{2}+C_{3}+ C_{4}+C_{5}+C_{6}+C_{7}+C_{8}+C_{9}\} \\
     \mathcal{T}_{\mathcal{A}}\otimes \mathcal{A}  \cong  \mathcal{A}\otimes \mathcal{T}_{\mathcal{A}}  &=&\{C_{1}, \ C_{2}, \ C_{3}, \ C_{4}+C_{7}, \ C_{5}+C_{8}, \ C_{6}+C_{9}\}  \\
  \mathcal{A} \wr \mathcal{A}  \cong (I \otimes \mathcal{A})\wr (\mathcal{A} \otimes I) &=&\{C_{1}, \ C_{2}, \ C_{3}, \ C_{4}+C_{5}+C_{6}, \ C_{7}+C_{8}+C_{9}\} \\
 \mathcal{A} \wr \mathcal{T}_{\mathcal{A}}  \cong (I \otimes \mathcal{A})\wr \mathcal{T}_{(\mathcal{A} \otimes I)}  &=&\{C_{1}, \ C_{2}, \ C_{3}, \ C_{4}+C_{5}+C_{6}+C_{7}+C_{8}+C_{9}\} \\
     \mathcal{T}_{\mathcal{A}} \otimes \mathcal{T}_{\mathcal{A}}&=&\{C_{1}, \ C_{2}+ C_{3}, \ C_{4}+C_{7}, \  C_{5}+C_{6}+C_{8}+C_{9}\} \\
     H(2,\mathcal{T}_{\mathcal{A}})&=&\{C_{1}, \ C_{2}+ C_{3}+C_{4}+C_{7}, \  C_{5}+C_{6}+C_{8}+C_{9}\} \\
     \mathcal{T}_{\mathcal{A}} \wr \mathcal{T}_{\mathcal{A}}&=&\{C_{1}, \ C_{2}+ C_{3}+C_{5}+C_{6}+C_{8}+C_{9}, \ C_{4}+ C_{7}\}  \\
     H(2,\mathcal{A})&=&\{C_{1}, \ C_{2}+ C_{4}, \ C_{3}+C_{7}, \ C_{5}, \ C_{6}+C_{8}, \ C_{9}\} \\
     (I\otimes\mathcal{T}_{\mathcal{A}})\wr (\mathcal{A}\otimes I) \cong  \mathcal{T}_{\mathcal{A}} \wr \mathcal{A}   &=&\{C_{1}, \ C_{2}+ C_{5}+C_{8}, \ C_{3}+C_{6}+C_{9}, \ C_{4}+C_{7}\} \\
\end{array} 
\]

The generalized Hamming scheme $H(2,\mathcal{A})$ is one of the above $13$ fusions that exist for all rank $3$ symmetric association schemes $\mathcal{A}$. For asymmetric rank $3$ association schemes $\mathcal{A}$, we get three additional fusions of $\mathcal{A} \otimes \mathcal{A}$ along with the $13$ fusions listed above. They are listed below,

\begin{enumerate}
    \item $\{C_{1}, \ C_{2}+ C_{7}, \ C_{3}+C_{4}, \ C_{5}+C_{9}, \ C_{6}, \ C_{8}\}$, 
     \item $\{C_{1}, \ C_{2}+ C_{3}, \ C_{4}+C_{7}, \ C_{5}+C_{9}, \ C_{6}+C_{8}\}$, 
      \item $\{C_{1}, \ C_{2}+ C_{3}+C_{4}+C_{7}, \ C_{5}+C_{9}, \ C_{6}+C_{8}\}$.
\end{enumerate}

An interesting project here would be to determine the strongly-regular graphs for which $\mathcal{A} \otimes \mathcal{A}$ has unexpected fusions, that is, fusions other than the ones listed above. Since, we have dealt with $H(2,\mathcal{A})$ here, the problem is down to two main cases, $\mathcal{A} \otimes \mathcal{A}$ and $\mathcal{A} \wr \mathcal{A}$.

\bibliographystyle{plain}
\bibliography{sample}

\end{document}